\documentclass{conm-p-l}

\usepackage{amssymb}
\usepackage{amsmath}
\usepackage{amsfonts}
\usepackage{graphicx,caption}
\usepackage{setspace}
\usepackage{amsthm}
\usepackage{fullpage}
\usepackage[autostyle=try]{csquotes}

\newcommand{\complex}{\mathbb{C}}

\newcommand{\rsphere}{ {\mathbb C_\infty} }
\newcommand{\disk}{{\mathbb D}}

\newcommand{\ucirc}{ {\mathbb S}}

\newcommand{\0}{\emptyset}
\newcommand{\lam}{{\mathcal L}}

\newcommand{\cl}{\overline}

\newtheorem{thm}{Theorem}[section]
\newtheorem{lem}[thm]{Lemma}
\newtheorem{cor}[thm]{Corollary}
\newtheorem{prop}[thm]{Proposition}
\theoremstyle{definition}

\newtheorem{defn}[thm]{Definition}

\theoremstyle{remark}
\newtheorem{rem}[thm]{Remark}
\numberwithin{equation}{section}

\begin{document}

\title{Unicritical and Maximally Critical Laminations}



\author{Brittany E. Burdette}
\address{Mathematics - UAB Birmingham AL 35294-0001} 
\email[Brittany E. Burdette]{beburd@uab.edu}
\author{Caleb S. Falcione}
\address{Mathematics - UAB Birmingham AL 35294-0001} 
\email[Caleb Falcione]{caleb.falcione@gmail.com}
\author{Cameron G. Hale}
\address{Mathematics - UAB Birmingham AL 35294-0001} 
\email[Cameron G. Hale]{CameronGHale@outlook.com}
\author{John C. Mayer}
\address{Mathematics - UAB Birmingham AL 35294-0001} 
\email[John C. Mayer]{jcmayer@uab.edu}



\date{Version of \today}

\begin{abstract}
We study the correspondence between unicritical laminations and maximally critical laminations with rotational and identity return polygons. Laminations are a combinatorial and topological way to study Julia sets. Laminations give information about the structure of parameter space of degree $d$ polynomials with connected Julia sets. 
\end{abstract}

\maketitle

\section{INTRODUCTION} 
\subsection{Motivation}
Definitions and facts in this section are adapted from Milnor \cite{Milnor:2006}. Let $P: \complex\to\complex$ be a polynomial of degree $d\ge 2$: $$P(z)=a_dz^d+a_{d-1}z^{d-1}+\dots+a_1z+a_0$$For $z_0 \in\complex$, we define the {\em iterates} of $z_0$ to be:  $$P^n(z_0) := P(P^{n-1}(z_0)).$$
We define the {\em forward orbit} of $z_0$ to be the set:
$$\{z_0,P(z_0),P^2(z_0),\dots,P^n(z_0),\dots\}.$$Compactify $\complex$ to the Riemann sphere $\rsphere$ by adding the point at infinity. For a polynomial, $P$, $\infty$ is an attracting fixed point: for $z\in\complex$ with $|z|$ sufficiently large, $\lim_{n\to\infty} P^n(z)=\infty$. The {\em basin of attraction} of $\infty$ is:  $$B_\infty:=\{z\in\complex\mid \lim_{n\to\infty} P^n(z)=\infty\}$$
By continuity, $B_\infty$  is an open set. Let $P$ be a polynomial of degree $d\ge 2$. We define the following sets in $\complex$ associated with $P$:
\begin{enumerate}
    \item The {\em Julia set} $J(P)$ := boundary of $B_\infty$.
    \item  The {\em Fatou set} $F(P)$:= $\rsphere\setminus J(P)$.
    \item The {\em filled Julia set}  $K(P):= \complex\setminus B_\infty$.
\end{enumerate}

It is known that $J(P)$ is nonempty, compact, and perfect; $K(P)$ is full (does not separate $\complex$), but $J(P)$ may separate $\complex$. We will only consider cases when $J(P)$ is connected. Periodic points of $P$ can be attracting, repelling, or indifferent. Attractive periodic orbits are in the Fatou set. Repelling periodic orbits are in the Julia set. Components of the Fatou set are called Fatou domains. There is only one unbounded Fatou domain whereas there can be zero, one, or countably infinitely many bounded Fatou domains.

An important part of complex dynamics is understanding the structure of parameter spaces. For example, the {\em unicritical degree} $d$ {\em connectedness locus} would be defined as such: $$\mathcal{M}_d=\{c\in \complex \mid \text{$J(P_c)$\ is\ connected}\}$$ where $P_c = z^d + c.$

The unicritical degree $d$ connectedness locus is a way of studying the unicritical polynomials of degree $d$ with connected Julia sets all at once.  There is also a multi-critical degree $d$ connectedness locus. Consider the family of monic and centered polynomials of degree $d$. 
$$ \mathcal{P}_d:=\{(a_0, a_1, \dots, a_{d-2})\in\complex^{d-1} \mid P(z)=z^d+0z^{d-1}+a_{d-2}z^{d-2}+ \dots+ a_0\}$$
The \emph{degree $d$ connectedness locus} is: $$\mathcal{C}_d = \{(a_0, a_1, \dots, a_{d-2})\in\complex^{d-1} \mid P\in \mathcal{P}_d   \text{ and\ $J(P)$\ is\ connected}\}$$ We study one of the connections between these two parameter spaces.

\subsection{Laminations}

Laminations are a combinatorial and topological way to study Julia sets. Unicritical laminations (Definition \ref{unicritical}) appear, for example, on the boundary of the main cuboid for cubic polynomials. A rotational polygon in a lamination corresponds to a fixed point in the Julia set, and an identity return polygon corresponds to a branch point in the Julia set that returns with no rotation.

Most of the definitions and facts in this section come from \cite{Blokh:2011}, \cite{mimbs:2013}, \cite{Cosper:2016}, and \cite{Schleicher:2009}. 

\begin{defn}[$\sigma_d$ map]
We use ``d-nary" coordinates on the circle $\ucirc$.
The map $\sigma_d : \ucirc \to \ucirc$ is defined to be $\sigma_d (t) = dt \pmod{1}$.  For example, $\sigma_2$ has binary coordinates, $\sigma_3$ has ternary, $\sigma_4$ has quarternary, and so on.
\end{defn}

\begin{figure}
    \centering
    \includegraphics[width = 2.1in]{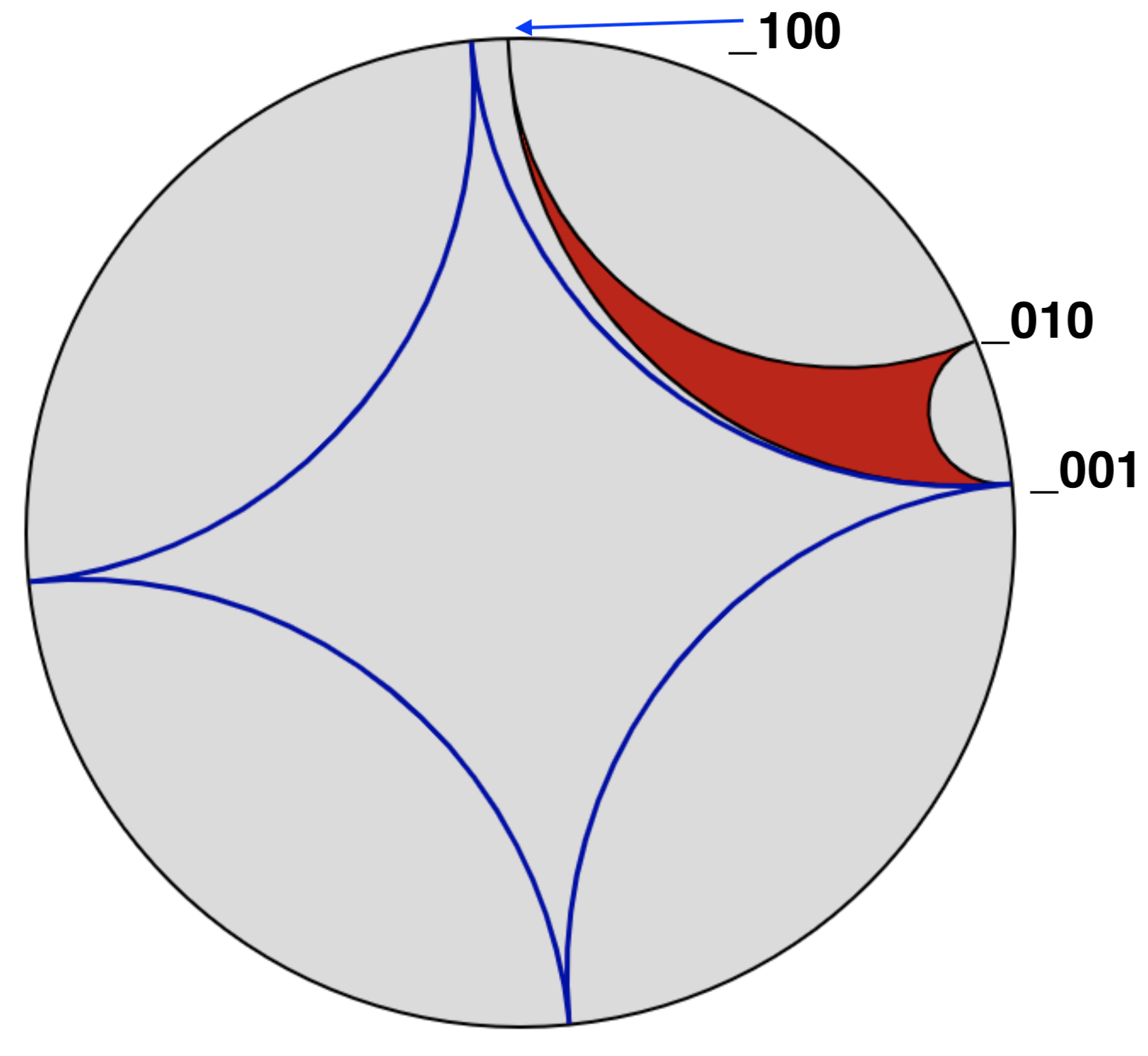}
    \includegraphics[width = 2in]{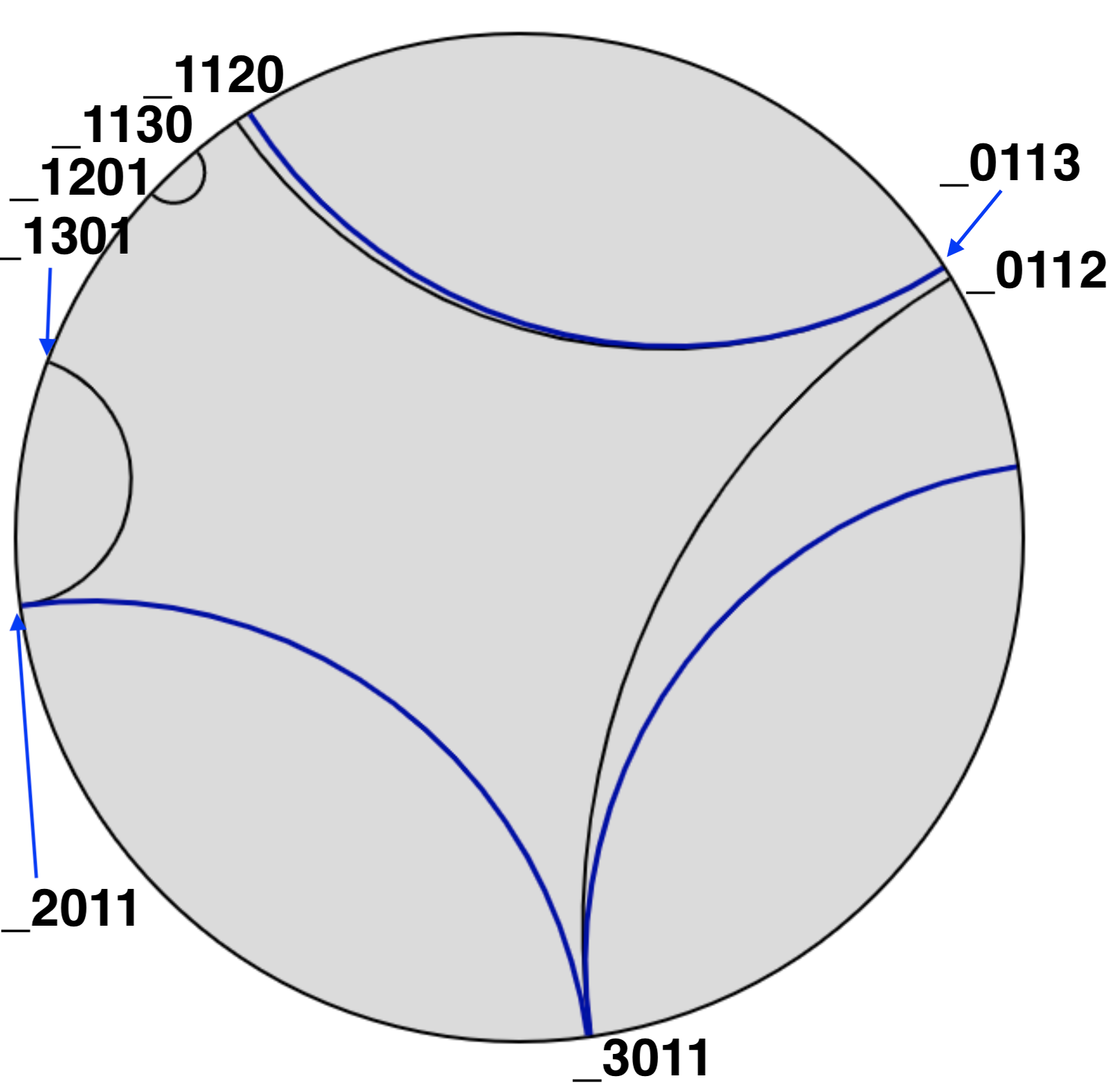}
    \caption[$d$-nary Coordinates]{
    Here we have two different examples of $d$-nary (in this case 4-nary) coordinates. On the left, we only use the digits 0 and 1, but on the right we use all 4 digits: 0, 1, 2, and 3. The underscore that precedes the digits indicates that those digits are repeated. We generate all lamination pictures except Figure \ref{Co-root Diagram} using Lamination Builder \cite{Falcione:2019}}.
    \label{fig:Coordinates}
\end{figure}

\begin{defn}[Lamination]
A {\em lamination}, $\lam$, is a collection of chords of the closed unit disk, $\cl\disk$, which we call {\em leaves}, such that:
\begin{enumerate}
    \item Any two leaves of $\lam$ meet, if at all, in a point of $\ucirc$.
    \item  $\lam^*:=\ucirc\cup\{\cup\lam\}$ is a closed subset of $\cl\disk$.
\end{enumerate}
In the case that condition (1) only is met, we call our collection of chords a {\em pre-lamination}.
\end{defn}

\begin{defn}[Leaf Mapping]
We denote mapping a leaf, ${\ell} = \cl{ab}$, with endpoints $a$ and $b$, under the map $\sigma_d$ by $$\sigma_d (\ell) = \cl{\sigma_d(a) \sigma_d(b)}$$

\end{defn}

\begin{defn}[Critical Chords]
A chord, $\ell = \cl{ab}$, is called {\em critical} when both its endpoints map to a single point $\sigma_d(a) = \sigma_d(b)\in\ucirc$. 
\end{defn}

\begin{defn}[Sibling Leaves]
Let $\ell_1\in\lam$ be a  leaf  and suppose
$\sigma_d(\ell_1)=\ell'$, for some non-degenerate leaf $\ell'\in\lam$.
A  leaf
$\ell_2\in\lam$, disjoint from $\ell_1$, is called a {\em sibling} of
$\ell_1$ provided $\sigma_d(\ell_2)=\ell'=\sigma_d(\ell_1)$.  A collection
${\mathcal S}=\{\ell_1, \ell_2, \dots, \ell_d\}\subset\lam$ is called a
{\em full sibling collection} provided that for each $i$,
$\sigma_d(\ell_i)=\ell'$ and for all $i\not=j$,
$\ell_i\cap\ell_j=\0$.
\end{defn}
\begin{figure}
    \begin{center}
     \includegraphics[width = 2in]{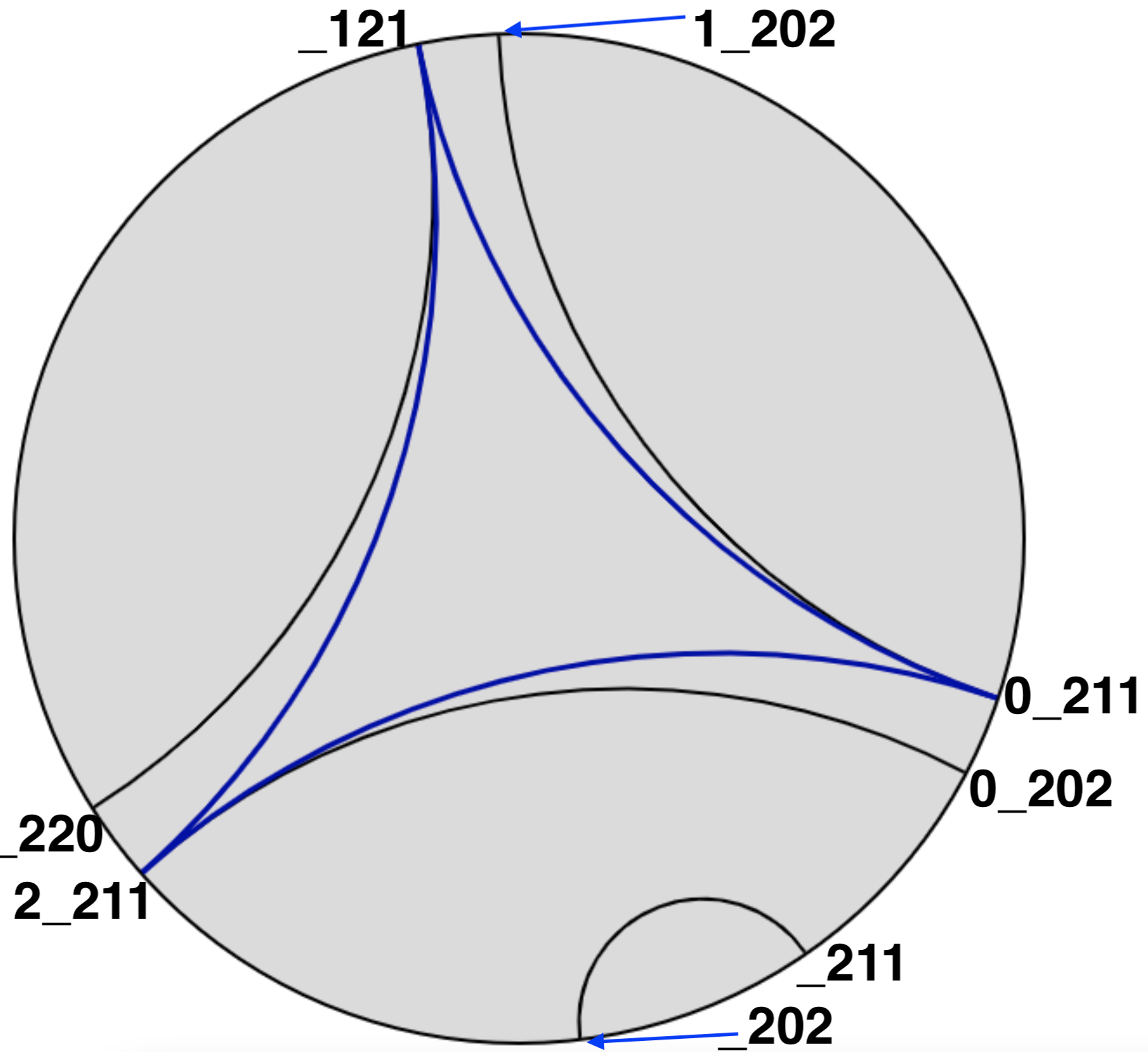}
    \caption[Full Sibling Collection]{Here we have a full sibling collection under $\sigma_3$. All of the long ``major" leaves map to one short ``minor" leaf. The central triangle is made of three critical chords creating an all critical triangle.}
\label{SibLam}
    \end{center}
\end{figure}

\begin{defn}[Sibling Invariant Lamination]\label{def: Sibling invaritant lam}
A lamination $\lam$ is said to be {\em sibling $d$-invariant} (or simply {\em invariant} if no confusion will result) provided that the following three statements hold:
\begin{enumerate}
\item (Forward Invariant) For every $\ell\in\lam$, $\sigma_d(\ell)\in\lam$.
\item (Backward Invariant) For every non-degenerate $\ell'\in\lam$,
there is a leaf $\ell\in\lam$ such that $\sigma_d(\ell)=\ell'$.
\item (Sibling Invariant) For every $\ell_1\in\lam$ with $\sigma_d(\ell_1)=\ell'$,
a non-degenerate leaf, there is a  full sibling collection $\{\ell_1, \ell_2, \dots, \ell_d\}\subset\lam$ such that  $\sigma_d(\ell_i)=\ell'$.
\end{enumerate}
\end{defn}

\begin{rem}
A sibling $d$-invariant lamination induces an equivalence relation. Two points on $\ucirc$ are equivalent if they are joined by a finite concatenation of leaves. We consider laminations for which this results in a closed equivalence relation. Thus, the sibling invariant laminations we will be considering have a fourth condition from \cite{mimbs:2013} not listed in the definition:
\begin{enumerate}
    \item[(4)] $\lam$ has finite equivalence classes, and all leaves are boundary chords of the convex hulls of equivalence classes.
\end{enumerate}
\end{rem}


\begin{defn}[Gap] A {\em gap} in a lamination, $\lam$,is the closure of a component of $\cl\disk\setminus\lam^*$.  A gap is {\em critical} iff two points in its boundary map to the same point. A gap with finitely many leaves in its boundary is usually called a {\em polygon}.  The leaves bounding a finite gap are called the {\em sides} of the polygon.  
\end{defn}

\begin{defn}[Fatou gap]
A {\em Fatou gap} in a lamination is a gap whose boundary intersected with $\mathbb{S}$ contains a Cantor set. 
\end{defn}

We can see Fatou gaps in a lamination in Figure \ref{CompRabbit}. The white spaces in the lamination are Fatou gaps since as we pull back our polygons farther, we get a gap whose boundary meets $\mathbb{S}$ in a Cantor set. 

\begin{rem}[Degree of a Fatou Gap]\label{degree fatou}
    The degree of a Fatou gap is the amount of criticality associated with the gap. For example, one critical chord in the gap gives a degree two Fatou gap which means the gap will map forward two-to-one.
\end{rem}

A Fatou gap in a lamination (that corresponds to a Julia set) maps to the closure of a bounded Fatou domain in dynamical space. Comparing the Rabbit lamination to the Rabbit Julia set in Figure \ref{CompRabbit}, one can see the correspondence.

\begin{defn}[All Critical Polygon]
A polygon whose endpoints all map to a single point is called an {\em all critical polygon}.
See Figure \ref{SibLam}.
\end{defn}

\begin{prop}[Order Preserving \cite{mimbs:2013}]
On both finite and infinite gaps of a $d$-invariant lamination, the map $\sigma_d$ preserves circular order.
\end{prop}

\begin{prop}[How Fatou Gaps Meet] \label{prop: gap meet}
Fatou gaps in a lamination can meet in two ways. Either they meet at a leaf shared in both of their boundaries, or a leaf in one gap meets another leaf in the other gap at a single point. 
\end{prop}

\begin{proof}
Suppose we have two gaps $G_1$ and $G_2$. The boundary of a gap is the convex hull of where it meets the circle. Note the convex hull of a gap cannot contain, or cross, another convex hull. This means that two gaps can share a piece of their convex hull such as a leaf, or a point joined by two leaves. Note that two gaps cannot meet at a point that is not an endpoint of a leaf for then the two convex hulls would intersect in their interior.
\end{proof}

\begin{rem}
Finite gaps do not meet since each gap is its own equivalence class.
\end{rem}

\begin{figure}
    \centering
    \includegraphics[width = 2in]{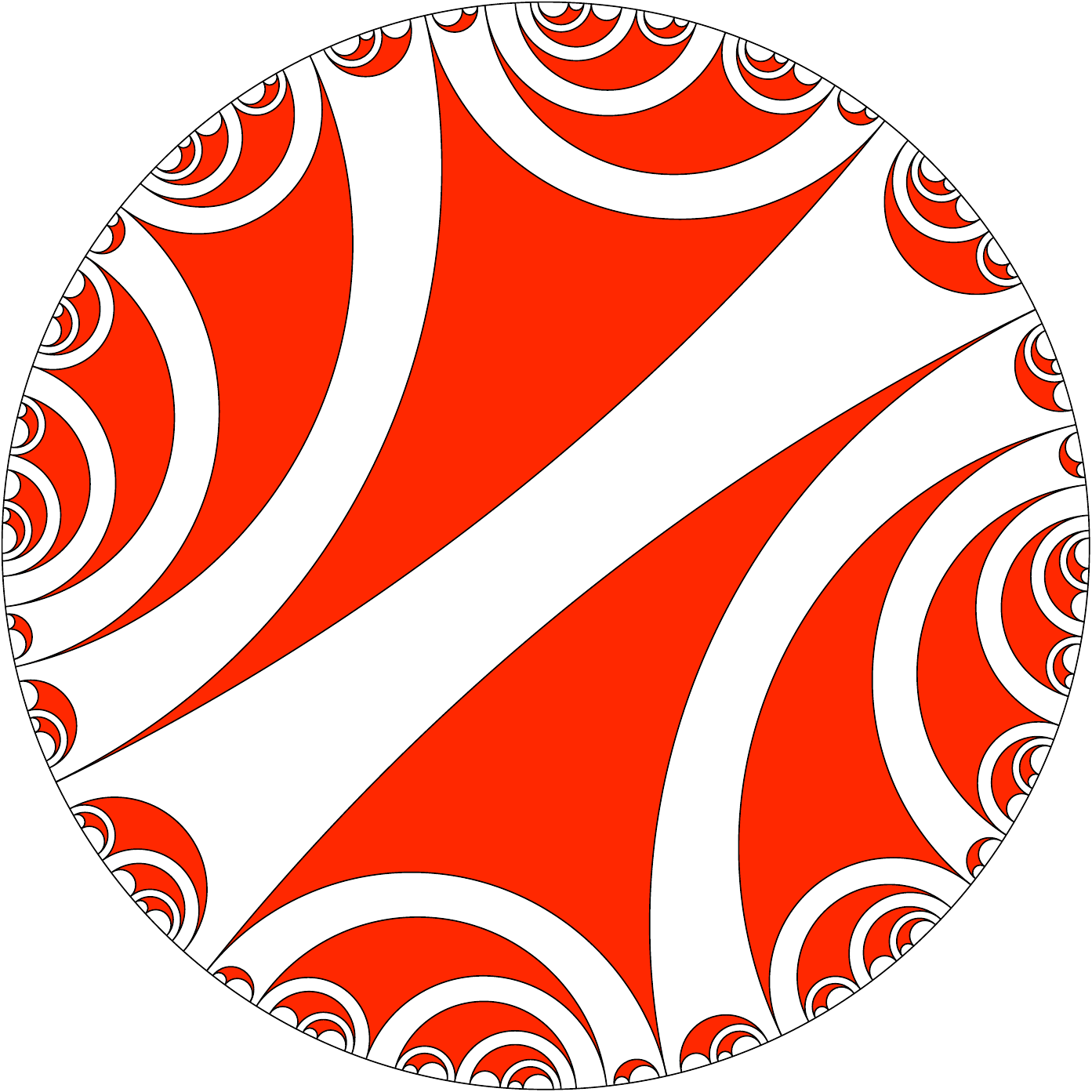}
    \hspace{.5in}
    \includegraphics[width = 2in]{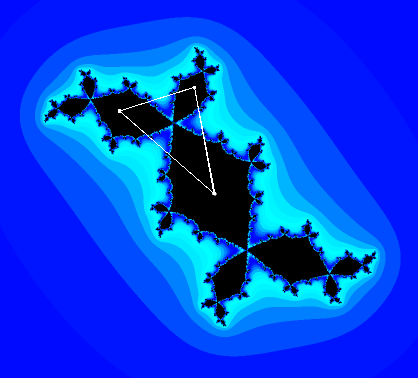}
    \caption[Rabbit Lamination and Julia Set]{Here are corresponding quadratic lamination (left) and Julia set (right) named the ``Douady Rabbit". There is a quotient map from the lamination to the Julia set. Under this quotient map the white gaps correlate with the black bounded domains of the Fatou set. The triangles map to ``pinch points" between the bounded Fatou domains.}
    \label{CompRabbit}
\end{figure}

\subsection{Pullback Laminations} \label{sec: pullback}

\begin{defn}[Critical Portrait]\label{def: Critical portriat}
A maximal collection of critical chords for $\sigma_d$ is called a {\em critical portrait} (maximal here meaning at least $d-1$ non-crossing critical chords). When a collection of critical chords meet at endpoints, the maximum may exceed $d-1$ by forming all critical polygons.
\end{defn}

\begin{defn}[Critical Sectors]

A critical sector is a region $C$ in the closed unit disk that is bounded by critical chords and arcs of the circle such that the boundary of $C$ maps onto the circle with degree 1.
\end{defn}

\begin{defn}[Forward Invariant Set]
A {\em forward invariant set} consists of periodic leaves or polygons that map forward preserving circular order without intersecting.
\end{defn}

\begin{defn}[Compatible]
A  critical portrait, $C$, is {\em compatible} with a forward invariant set, $F$, when $F$ meets $\cup C$ only at endpoints of leaves or vertices of polygons.
\end{defn}

\begin{defn}[Branches of the Inverse]\label{def: branches of the inverse}
Let $C$ be a critical portrait. Every critical sector $S$ in $\cl\disk$ defined by $C$ will have a function $\tau: \partial \disk \to \partial S$ that is one to one, and $\sigma_d \circ \tau$ is the identity on $\partial \disk$. 
\end{defn}

Combining the previous definitions we can now define a pullback scheme.

\begin{defn}[Pullback Scheme]
Let $C$ be a critical portrait and $F$ a compatible forward invariant set. The corresponding collection $PB(F,C)$ of the branches of the inverse determined by $C$ as in Definition \ref{def: branches of the inverse} gives us a {\em pullback scheme} for $F$.
\end{defn}

Using Lamination Builder \cite{Falcione:2019}, we can visualize multiple steps of the pullback scheme. 

\begin{figure}
    \centering
    \includegraphics[width = 2in]{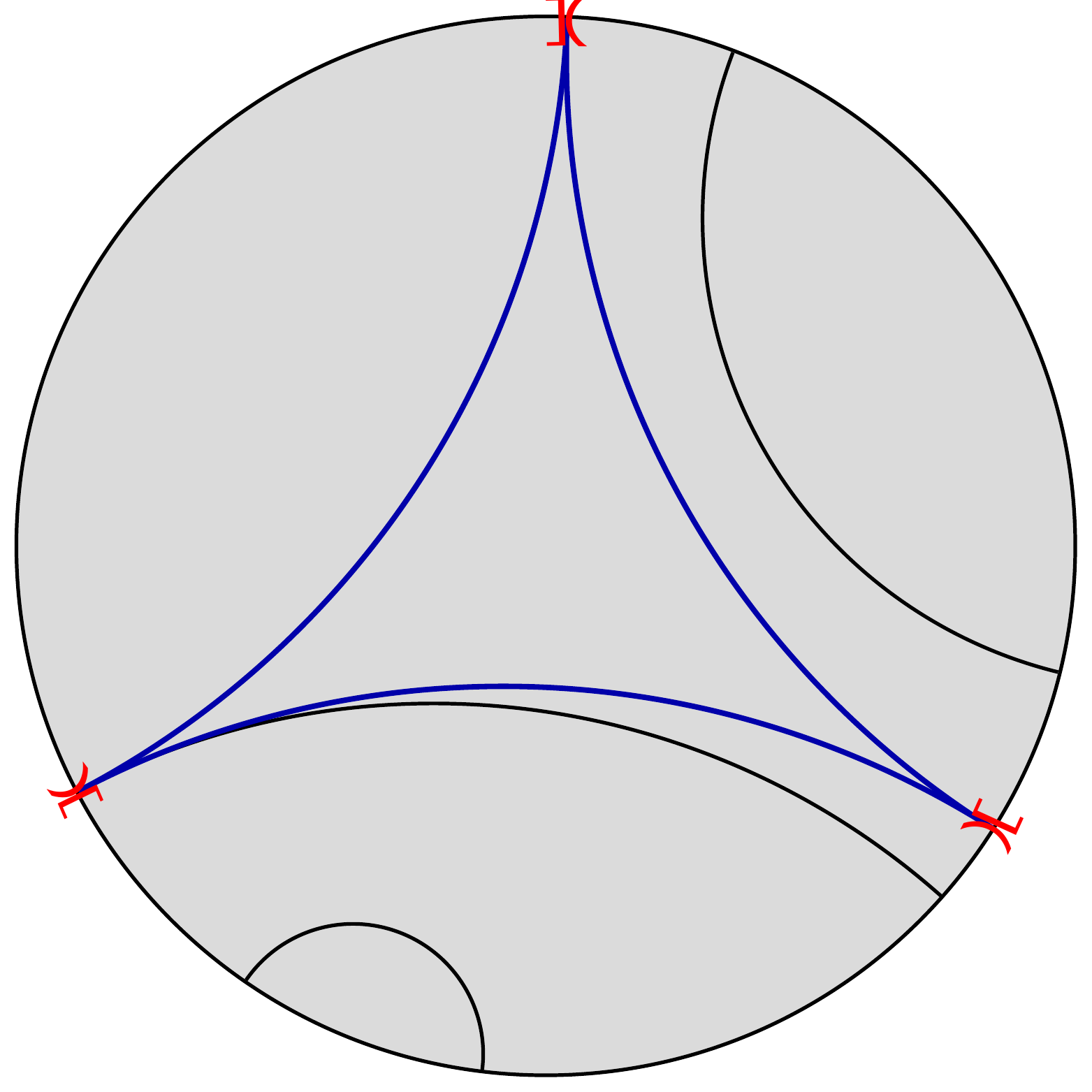}
    \hspace{.3in}
    \includegraphics[width = 2in]{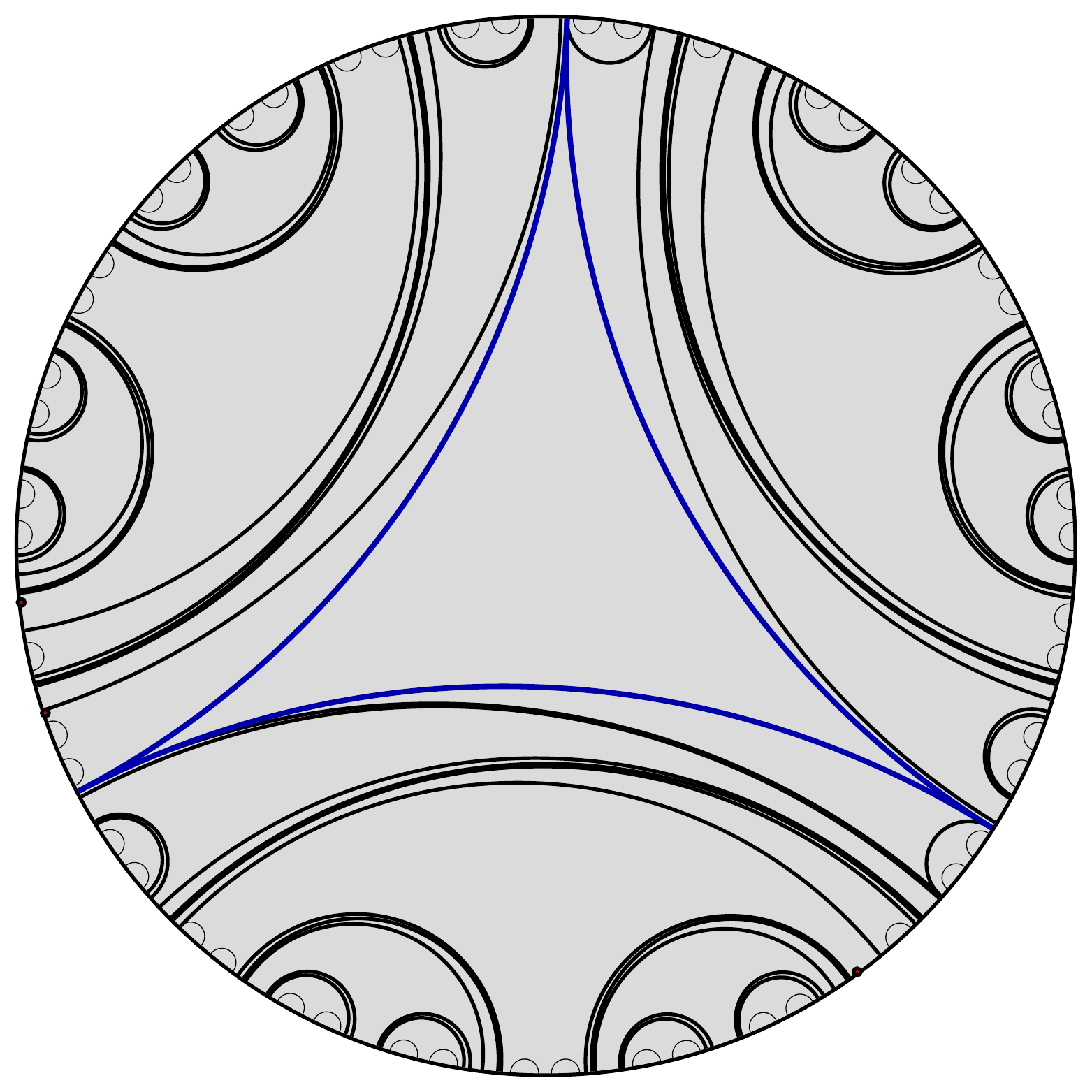}
    \caption[Branches of the Inverse]{(Left) A pullback scheme with an all critical triangle and included branches of the inverse marked by closed ends of half open intervals on the circle. (Right) The resulting pullback with the indicated pullback scheme.}
    \label{fig:Pullback Scheme}
\end{figure}

\begin{rem}
There can be multiple critical portraits compatible with $F$ that define the pullback scheme differently. There are some critical portraits that are compatible with $F$, but the related pullback lamination does not satisfy condition (4) of Definition \ref{def: Sibling invaritant lam}. The following theorem is well known.
\end{rem}

\begin{thm}
Let $F$ be a periodic forward invariant set under $\sigma_d$, $C$ a compatible critical portrait, and the pullback scheme $PB(F,C) = \{\tau_1, \tau_2, \dots, \tau_d\}$. Let $F_0 = F,$ and $F_1 = F_0 \cup \tau_1(F_0) \cup \tau_2(F_0) \cup \dots \cup \tau_d(F_0)$. In general, for any given stage $n$ of the pull back $F_n = F_{n-1} \cup \tau_1(F_{n-1}) \cup \dots \cup \tau_d(F_{n-1})$. Let $F_\infty  = \bigcup^\infty_{n=0} F_n$, and let $ \mathcal{L} = \cl{F_\infty}$. Then, $\mathcal{L}$ is a sibling $d$-invariant lamination.
\end{thm}

\begin{proof}[Proof Sketch]
Note, $F_0$ is forward invariant, and the $\tau_i$ are homeomorphisms away from critical values. Hence, each $F_i$ is forward invariant, and up to $F_{i-1}$ is backward and sibling invariant. It follows that $F_\infty$ is sibling $d$-invariant but is not closed. Taking the closure $\cl{F_\infty}$ picks up points in the circle and limit leaves. The only potential problem here is crossing limit leaves. But, if limit leaves were to cross when we take the closure, that means there were some crossings in $F_\infty$ prior to the closure. To ensure $\cl{F_\infty}$ has finite equivalence classes, one must choose the appropriate pullback scheme.  
\end{proof}

\begin{defn}[Grand Orbit] \label{def: grand orbit}
Let $P$ be an arbitrary subset of a lamination (for example a point, leaf, polygon, gap, etc.) in a $d$-invariant lamination. The {\em grand orbit} $\mathcal{GO}(P)$ is defined to be 
$$\mathcal{GO}(P) = \{ \sigma_d^{-n}(\sigma_d^k(P)) | n,k \in \mathbb{Z}^+ \cup {0} \}.$$
\end{defn}

 \subsection{Periodic Polygons} 
Periodic polygons (and leaves) in a lamination can be rotational, rotation return, or identity return. We define each type below.

\begin{figure}
    \centering
    \includegraphics[width = 1.9in]{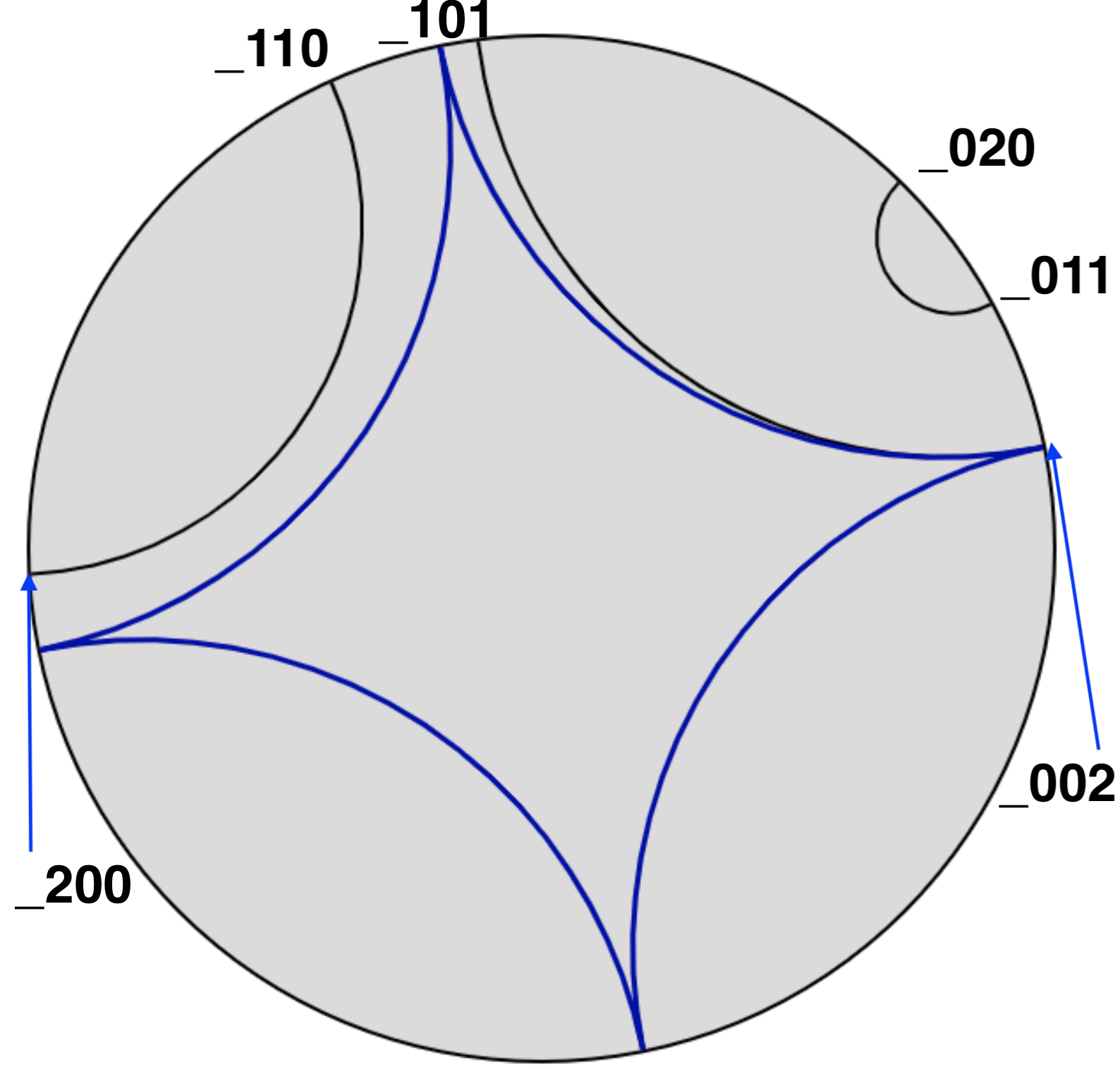}
    \includegraphics[width = 1.9in]{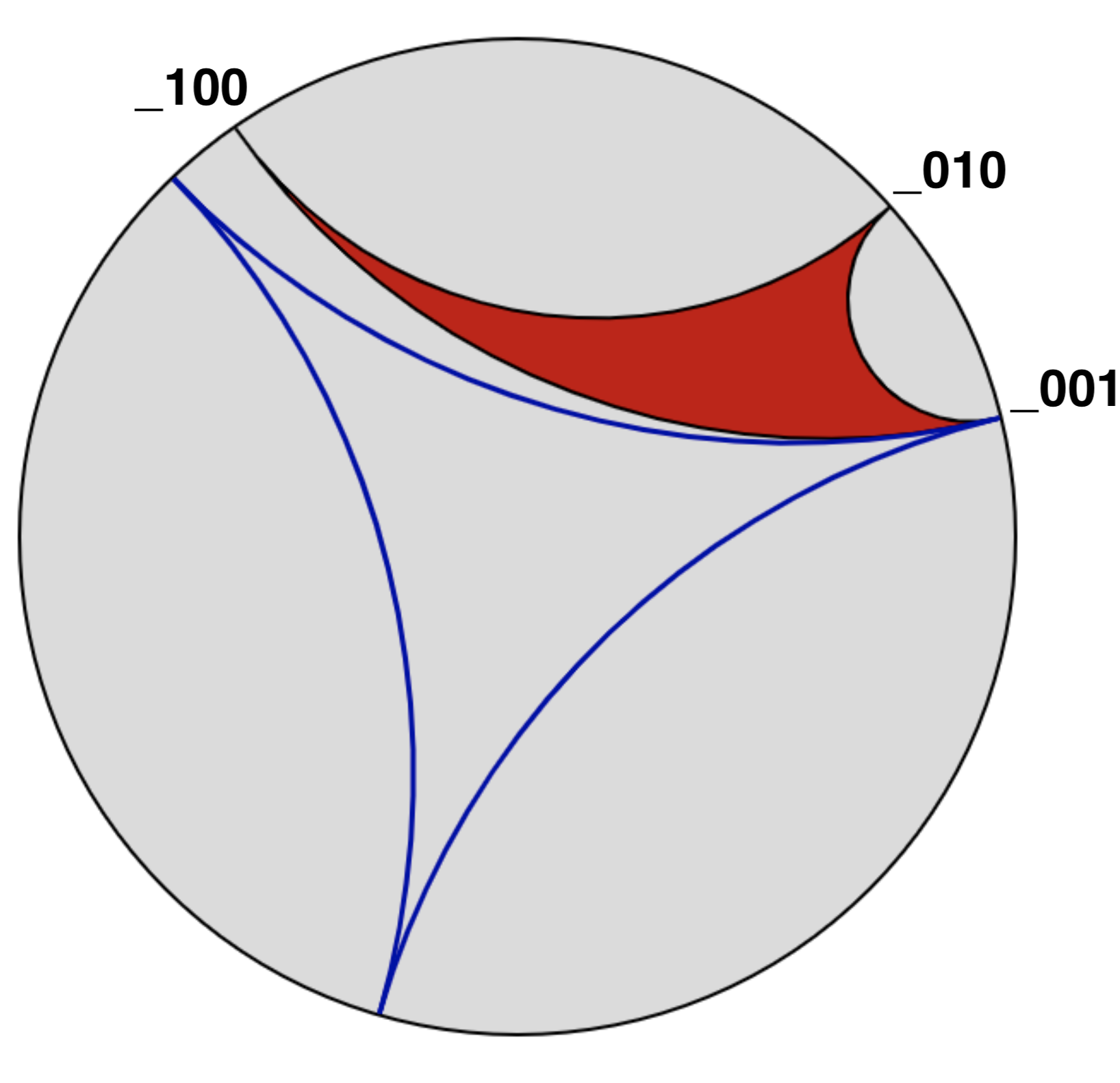}
     \caption{Identity return leaf in $\sigma_4$ (left), rotational polygon in $\sigma_3$ (right)}
    \label{fig:3 cases}
\end{figure}

 \subsubsection{Rotational Polygons}
 
 \begin{defn}[Rotational Set]
Consider $\sigma_n$: $\mathbb{S} \rightarrow \mathbb{S}$ for a particular $n \geq 2.$

Let $P = \{ x_i | 0 < x_1 < x_2 < x_3 < ... < x_k < 1 \}$ be a finite set in consecutive order in $\mathbb{S}$. We say that $P$ is a {\em rotational set} (for $\sigma_n$) iff \begin{enumerate}
    \item $\sigma_n(P)=P$, and
    \item For $1 \leq j \leq k$, if $\sigma_n(x_j) = x_i$, set $i(j)=i$. Then for all $j$, $j-i(j) (\text{mod }k)$ is the same. 
\end{enumerate}
If (2) holds (but possibly not (1)), we say $\sigma_n$ is circular order-preserving on $P$. We call a rotational set which is a single periodic orbit, a {\em rotational orbit}. 
\end{defn} 
 
 \begin{defn}[Rotation Number]
To each rotational set we can assign a {\em rotation number}, a rational number $0 \leq \dfrac{p}{q} < 1$ in lowest terms.  

Let $\mathcal{O} = \{x_1 < x_2 < x_3...<x_q \}$ be a rotational periodic orbit. Suppose that $\sigma_n(x_1)=x_j$. Set $p=j-1$. The rotation number of $\mathcal{O}$ is $\dfrac{p}{q}$. Our notation is $\rho(\mathcal{O}) = \rho(x_1)=\rho(x_i) = \dfrac{p}{q}$. 
\end{defn}

 \begin{defn}[Rotational Polygon] \label{rotational poly}
A polygon (or a leaf) $P=P_0$ is said to be {\em rotational} iff $\sigma_d|_{P_0}$ maps $P_0$ back to itself preserving circular order and $\forall i \neq j$,  $\sigma_d(p_i) = p_j$ according to some nonzero rotation number. 

\end{defn}
 
  \begin{defn}[Rotation Return Polygon]
 A polygon is said to be {\em rotation return} iff it is a polygon that maps off of itself and when it returns it has a nonzero rotation number. 
 \end{defn}

 \subsubsection{Identity Return Polygons}
  \begin{defn}[Identity Return]\label{identity return poly}
A polygon (or a leaf) $P=P_0$ is said to be {\em identity return} iff its {\em orbit}
$$\{P_0, P_1=\sigma_d(P_0), P_2=\sigma_d(P_1), P_3,\dots, P_n=P_0 \}$$ is periodic (of least period $n$) and has
the properties 
\begin{enumerate} \item  the polygons in the orbit are pairwise disjoint,  
\item $\sigma_d^{n}|_{P_0}$ is the identity (rotation number = 0), and
\item $P_i$ maps to $P_{i+1 \pmod n}$  preserving circular order.
\end{enumerate}
\end{defn}

Each vertex (and each side of $P$) is in a different orbit of period $n$ since the polygons in the orbit are disjoint. Restating Kiwi's theorem \cite{Kiwi:2001} for polygons in a lamination, we have:

If a periodic polygon $P$ in a $d$-invariant lamination has exactly $k$ distinct critical values associated with it, then the number of orbits of sides of $P$ is at most $ k$ + 1. Moreover, if the number of orbits of sides is $k + 1$, then $P$ is identity return.

Thus, $d$ is the maximal number of sides for an identity return polygon in a $d$-invariant lamination.

In the special case that the polygon has period $n=1$, we give the following definition:

\begin{defn}[Fixed Polygon]
A polygon (or a leaf) $P$ is said to be {\em fixed} iff $\sigma_d(P)=P$ and $\sigma_d(x_i)=x_i$ where $x_1, x_2, ..., x_i$ are the vertices of $P$ (endpoints of leaves of $P$).

\end{defn}

However, we only discuss correspondences with identity return and rotational polygons here.
\subsection{Leaf Length}

The following definitions are adapted from Cosper et al. \cite{Cosper:2016}.
\begin{defn}[Parameterizing the Circle]
The positive order on the circle $\ucirc$ will be fixed in the counterclockwise direction. Let $|(a,b)|$ be the length in the parameterization of the arc $(a,b)$ in $\ucirc$ from $a$ to $b$ counterclockwise. Given a chord $\cl{ab}$ there are two arcs of $\ucirc$ it subtends. Define the length $|\cl{ab}|$ of $\cl{ab}$ to be the shorter of the two arcs subtended. Note, the maximum length for a leaf is $\frac{1}{2}$.
\end{defn}

\begin{defn}[Leaf Length Function]
$$|\sigma_d(\cl{ab})| = \left\{ \begin{array}{cc}
     & d|\cl{ab}| \pmod1, \ \mathrm{if}\  d|\cl{ab}| \pmod1 \le \frac{1}{2} \\
     & 1 - d|\cl{ab}| \pmod1, \ \mathrm{if}\  d|\cl{ab}| \pmod1 \ge \frac{1}{2}
\end{array}\right.$$
\end{defn}

\begin{figure}
    \centering
    \includegraphics[width = 4in]{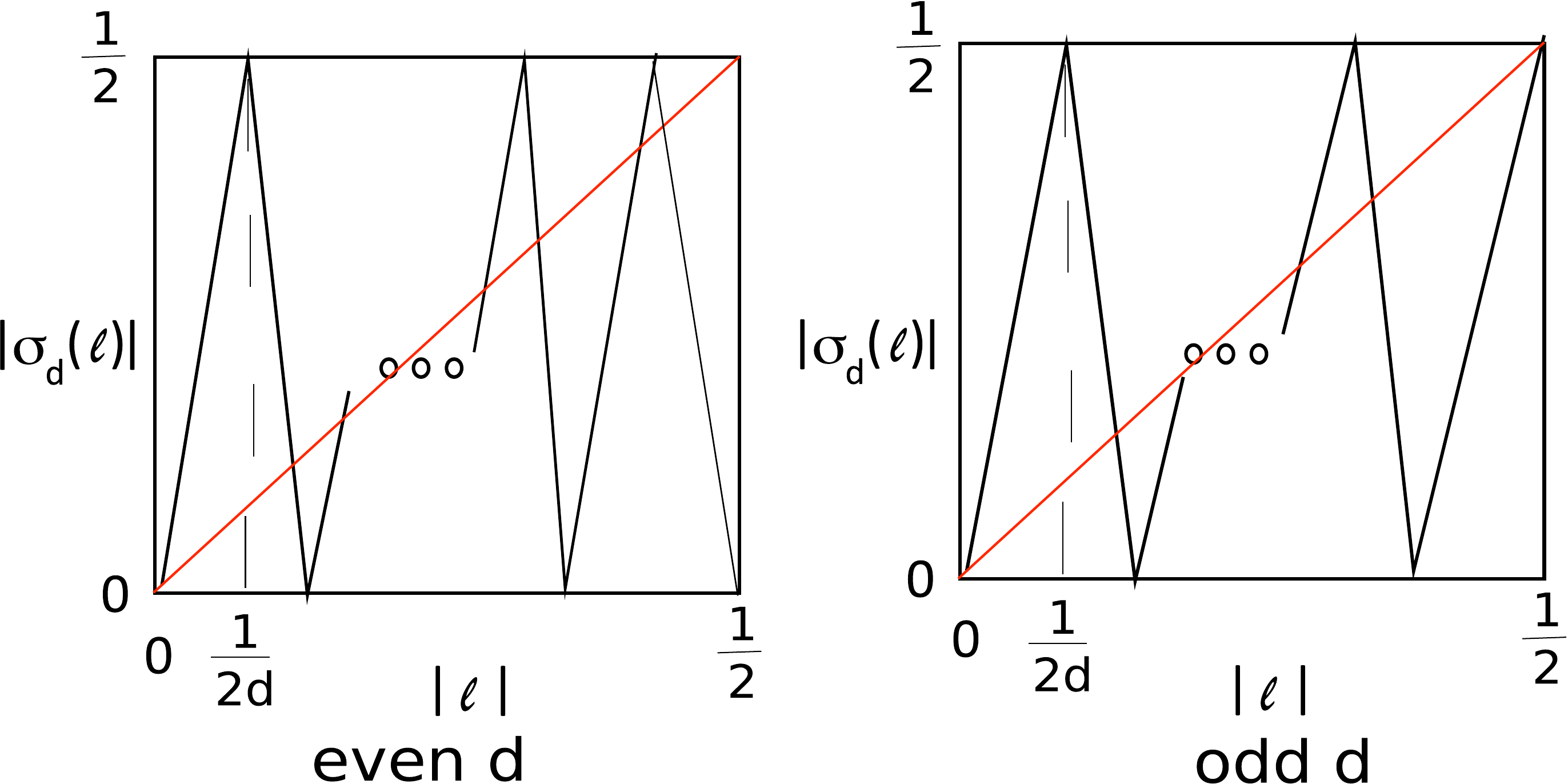}
    \caption[Leaf Length Graph]{Here is the leaf length graph for even and odd $d$ under $\sigma_d$ from Cosper et al. \cite{Cosper:2016}.}
    \label{fig: Leaf Length}
\end{figure}

\begin{prop}[Growing Leaves]\label{prop: leaf growth}
A leaf, $\ell$, of length $|\ell| < \frac{1}{d+1}$ will increase in length under $\sigma_d$ until $|\sigma_d^k(\ell)| \ge \frac{1}{d+1}$.
\end{prop}

\begin{rem}
Note that the first positive fixed point is $\frac{1}{d+1}$ which can be seen in Figure \ref{fig: Leaf Length} where the diagonal intersects the graph for the first time.
\end{rem}

\begin{defn}[Distance Between Leaves] \label{def: leaf distance}
 Take two leaves $\ell_1 = \cl{ab}$ and $\ell_2 = \cl{cd}$ with the order of the points being $a < b < c < d$  in a lamination $\mathcal{L}$. Define the distance between $\ell_1$ and $\ell_2$ to be $$dist(\ell_1, \ell_2) = |(b,c)| + |(d,a)|$$
\end{defn}

\section{Unicritical Laminations}

\begin{defn}[Unicritical Lamination] \label{unicritical}
A $d$-invariant lamination which is compatible with an all critical $d$-gon is called a {\em unicritical lamination}.
\end{defn}

\begin{defn}[Major and Minor Leaves]
For the orbit of a periodic leaf in a unicritical lamination, the leaf, $M$, closest to critical length is called the {\em major}, and the image $\sigma_d(M)$ of the major is called the {\em minor} usually denoted $m$.
\end{defn}

\subsection{Major All Critical (MAC)}

\begin{defn}[Major All Critical (MAC)]\label{def: MAC}
The major leaf in a periodic leaf orbit for $\sigma_d$ which is compatible with an all critical $d$-gon is called a {\em major all critical (MAC)} leaf. This MAC leaf may be a single leaf or a side of a polygon.
\end{defn}

\begin{defn}[Canonical MAC Lamination] \label{def: MAC lam}
Let $M$ be a MAC leaf. Attach a guiding all critical $d$-gon to one of the endpoints of $M$. Pullback $M$ with respect to the guiding all critical $d$-gon as described in Section \ref{sec: pullback}. We call this MAC lamination, $\mathcal{L}(M)$, or just $\mathcal{L}$ when $M$ is understood.
\end{defn}

\begin{rem}[First return of MAC leaf]
If a MAC leaf is just a leaf and not a side of a polygon, the first time the leaf returns to itself is by the identity except when the rotation number is $\dfrac{1}{2}$.

\end{rem}

\subsection{Single Critical Moment (SCM)}

\begin{defn}[Maximal Critical Sector] \label{def: maxsector}
Let $C$ be a critical portrait. If a critical sector of $C$ has in its boundary all critical leaves and a side of each all critical polygon, then we call it a {\em maximal critical sector}. See Figure \ref{fig:maximal critical sector}.
\end{defn}

\begin{figure}
    \centering
    \includegraphics[width = 1.9in]{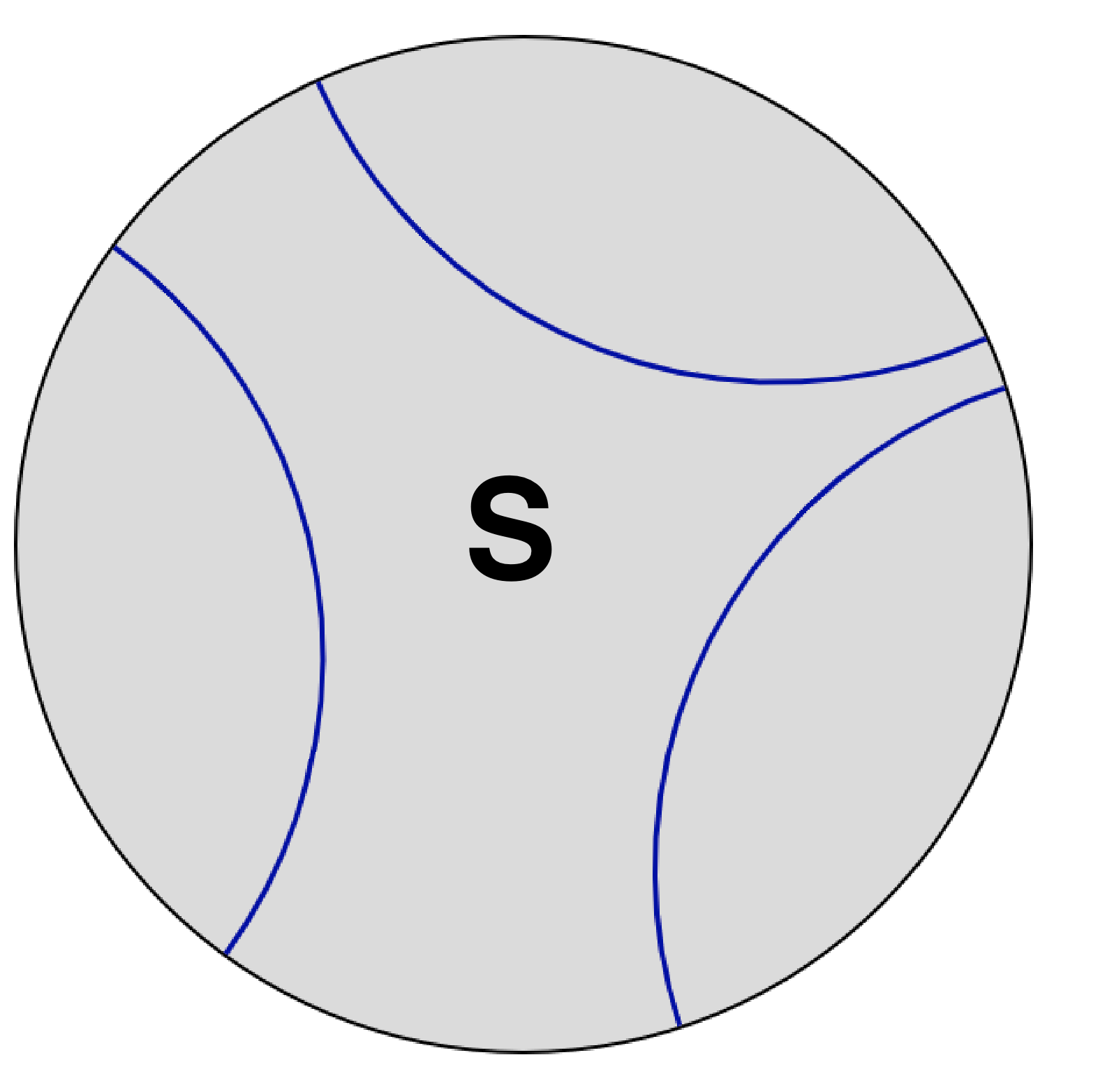}
    \includegraphics[width = 1.9in]{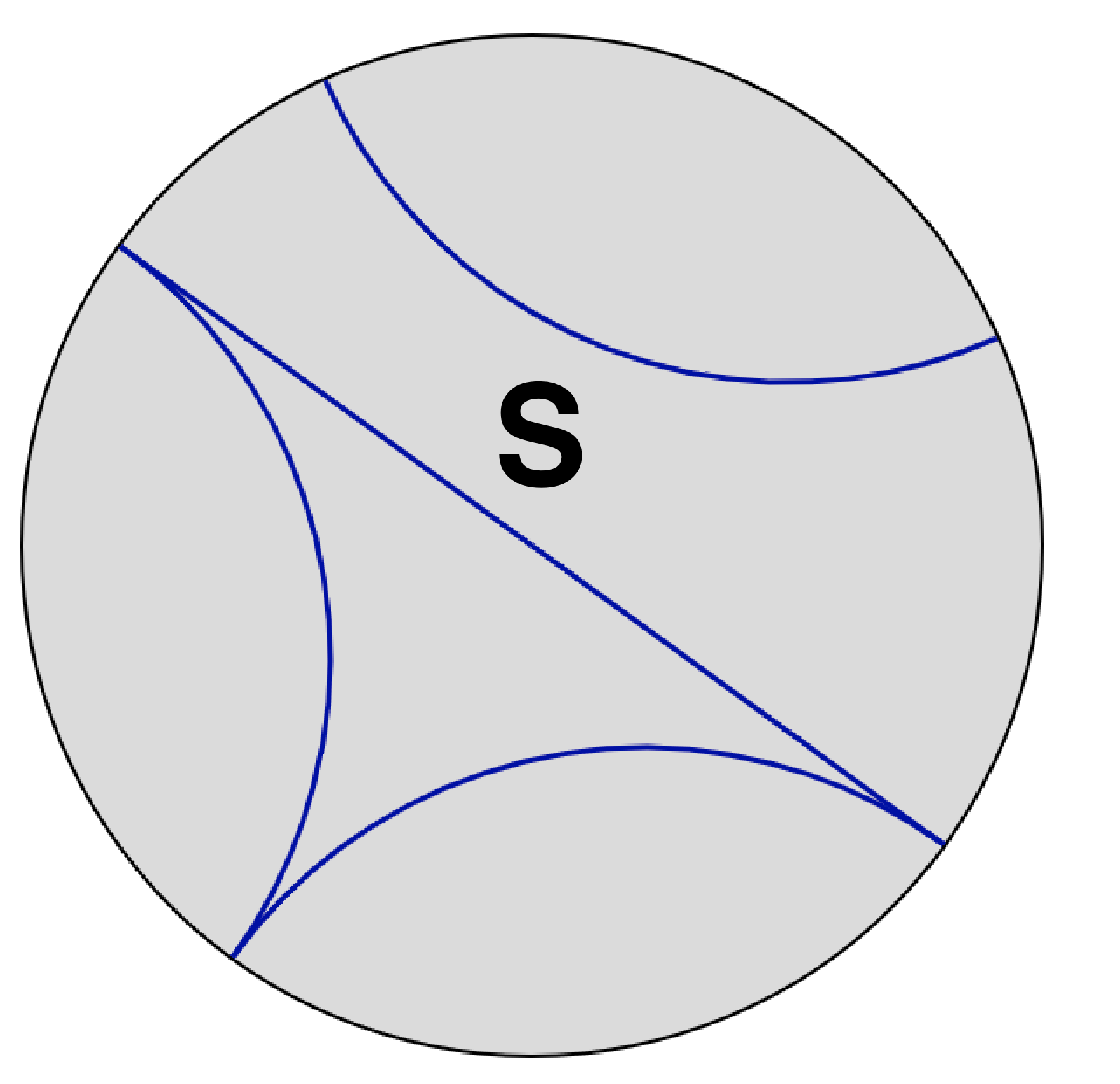}
    \includegraphics[width = 1.9in]{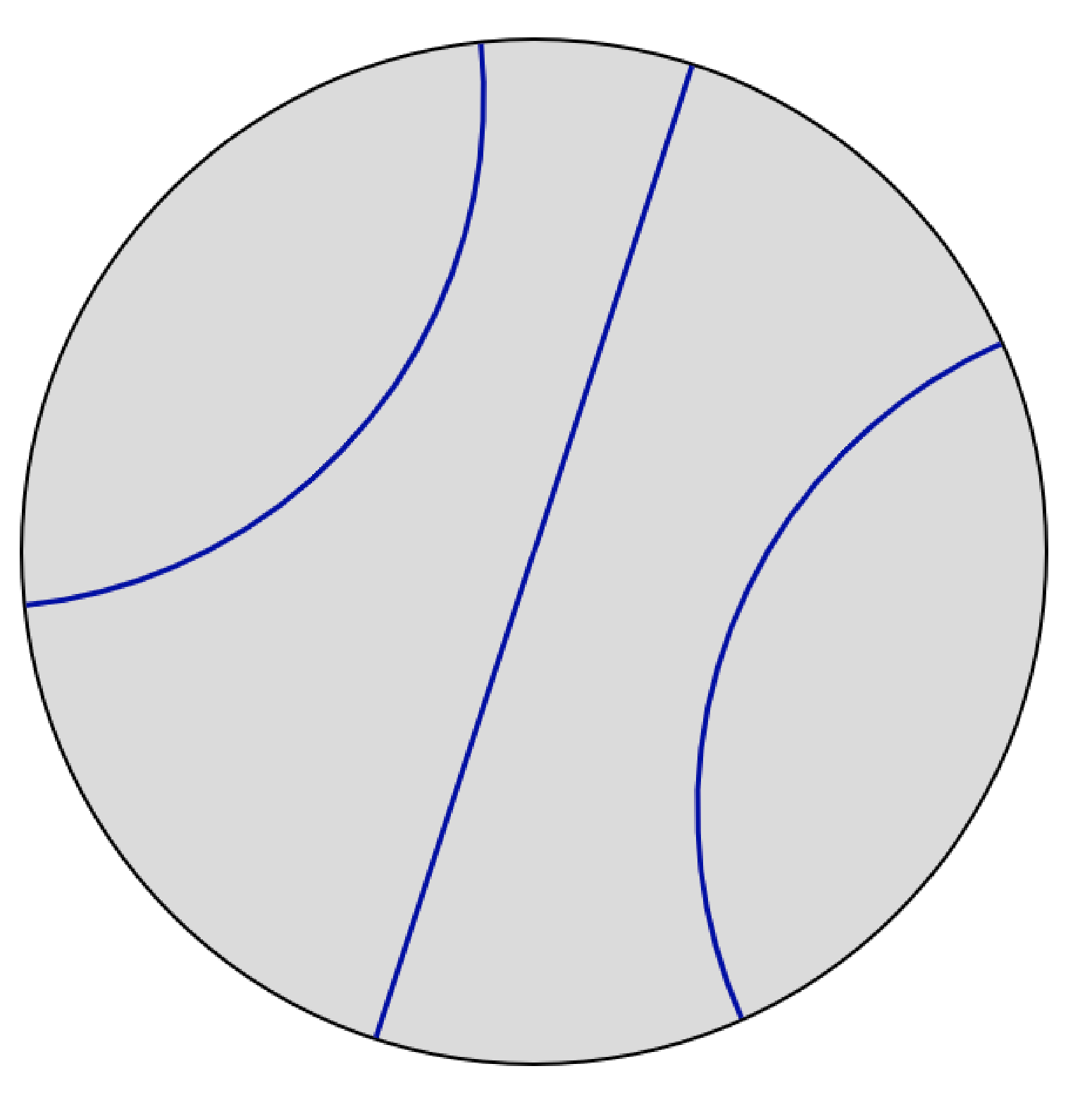}
    \caption{We can see on the left and middle figure that $S$ is a maximal critical sector as it has all critical leaves on its boundary and a side of the all critical polygon (if there is one). On the right, none of the four critical sectors will be maximal as no sector has all critical chords on its boundary.}
    \label{fig:maximal critical sector}
\end{figure}

\begin{defn}[Single Critical Moment (SCM)]\label{def:SCM}
A periodic return polygon, $P$, for $\sigma_d$ is {\em single critical moment (SCM)} if it has the following properties: 
\begin{enumerate}
    \item At one moment in its orbit $P$ must be inside a maximal critical sector.
    \item No side of $P$ makes a closer approach to criticality in a different sector.
\end{enumerate}
\end{defn}

\begin{defn}[Canonical SCM Lamination] \label{def: SCM lam}
Let $P$ be an SCM polygon and $\mathcal{C}$ be a collection of $d-1$ guiding critical chords each touching one endpoint of each of the $d-1$ longest sides of $P$. These longest sides of $P$ are required to be adjacent to each other. Pullback $P$ with respect to $\mathcal{C}$ as described in Section \ref{sec: pullback}. We call this the {\em canonical SCM lamination} for $P$.
\end{defn}

\begin{rem}
In Definition \ref{def: SCM lam}, the existence of the collection $\mathcal{C}$ of guiding critical chords follows immediately from Definition \ref{def:SCM}.
\end{rem}

\begin{defn}[Maximally Critical SCM Polygon] \label{def: max crit poly}
    A {\em maximally critical SCM polygon} is where criticality is broken up as much as possible: all Fatou gaps will be degree two. See Remark \ref{degree fatou}.
\end{defn}

We now consider two cases: polygons that return by the identity and polygons that are rotational. We are interested in finding the correspondence between polygons with SCM and MAC orbits. Examples of MAC and SCM laminations can be seen in Figures \ref{fig:canonical ir SCM} and \ref{fig:canonical rot SCM}.

In Barry's PhD thesis \cite{Barry:2015}, identity return triangles with a single critical moment were shown to stand in one to one correspondence with MAC leaves of a unicritical cubic lamination. Properties of these leaves and triangles were deduced. Our aim is to generalize this correspondence to certain $d$-invariant laminations for $d > 3$. In the identity return and rotational polygon cases, we will identify the SCM polygons that correspond one-to-one to MAC laminations.






There are two ways to define the period of an object: the first return of the vertices of the object vs. the first return of the object. If the polygon is identity return, then these are the same since the object and vertices return back to themselves at the same iterate. However, in the rotational and rotation return cases the period of the object and the period of the vertices differ. In fact, the period of the object is a factor of the period of the vertices. We will denote the period of the vertices as $k$, and the period of the object as $r$. In the rotational case, $r$ is always 1.

The following is the major theorem of this paper.

\begin{thm}[MAC $\Longleftrightarrow$ SCM] \label{thm: Main}
There exists a one to one correspondence between canonical MAC laminations and canonical SCM laminations for $\sigma_d$ as follows:
\begin{enumerate}
    \item MAC identity return leaves correspond to SCM identity return $d$-gons,
    \item MAC rotational polygons correspond to SCM rotational $k(d-1)$-gons,
    
\end{enumerate}
all with $d-1$ consecutive major leaves.
\end{thm}

In this section, we will develop an understanding of the structure of both SCM and MAC laminations as well as any similarities they might have between them. We establish one direction of Theorem \ref{thm: Main} in this section, and the other direction in the following section. 

\subsection{Endcaps}

We now want to define what we mean by \emph{endcaps} or \emph{endcaps intervals}. Assume we have a forward invariant set, $F$, and a compatible critical set, $C$. Take the major leaf from $F$. (In our cases, our major leaf is unique) Following the process from Section \ref{sec: pullback}, generate the first pullback to get the siblings of our initial leaf. One endpoint of each of the leaf and siblings will be touching a critical chord. An \textbf{endcap interval} is the interval from the endpoint of the leaf or sibling not touching a critical chord to the associated critical chord. In the special case of the symmetric sibling portrait, all siblings are the same length and each endcap interval is also the same length. Each endcap, where we start with a major leaf, maps one-to-one to the minor since end points of the interval are the end points of the major or one of its siblings.

\begin{figure}
    \centering
    \includegraphics[width=3in]{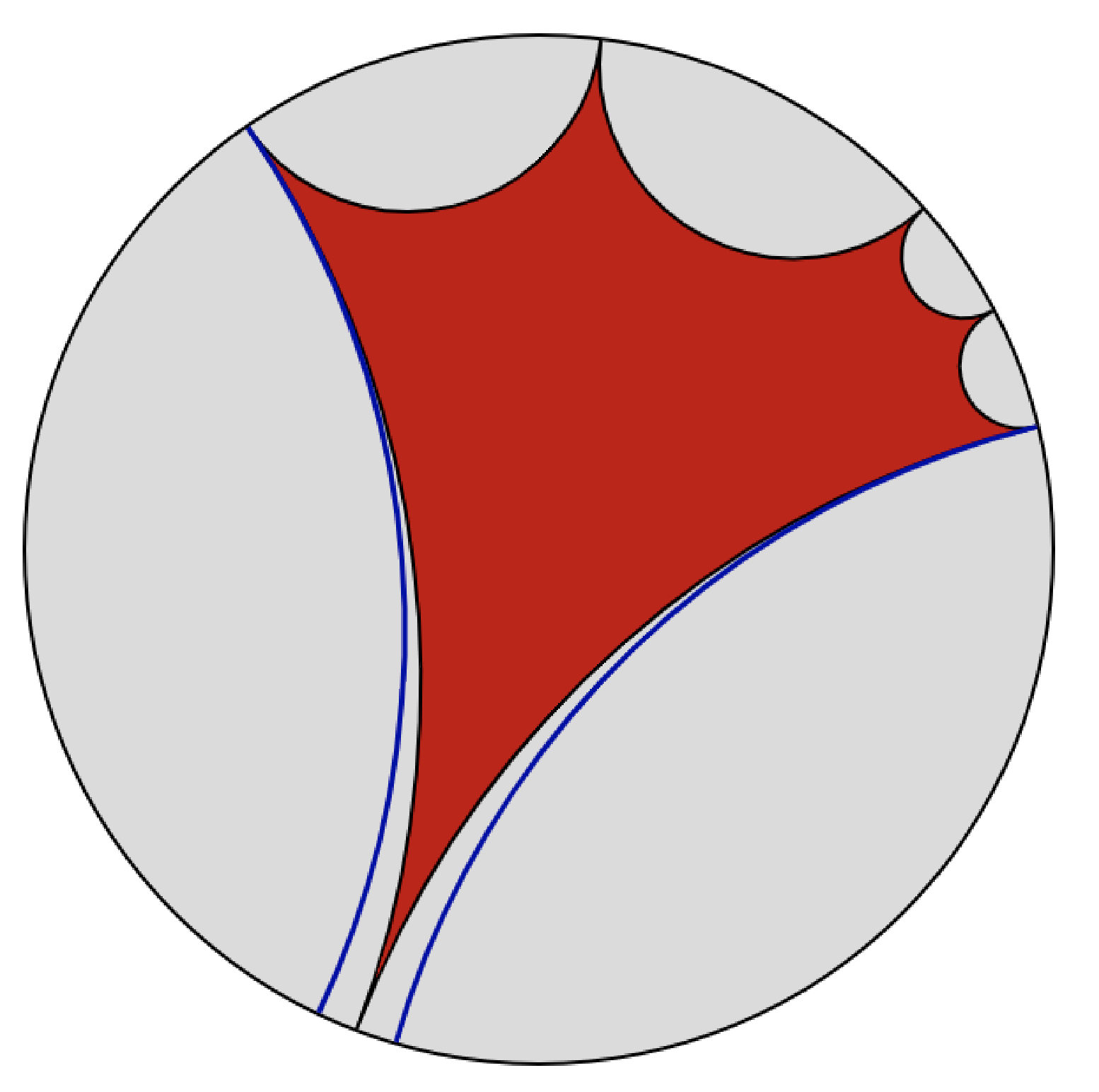}
    \includegraphics[width=3in]{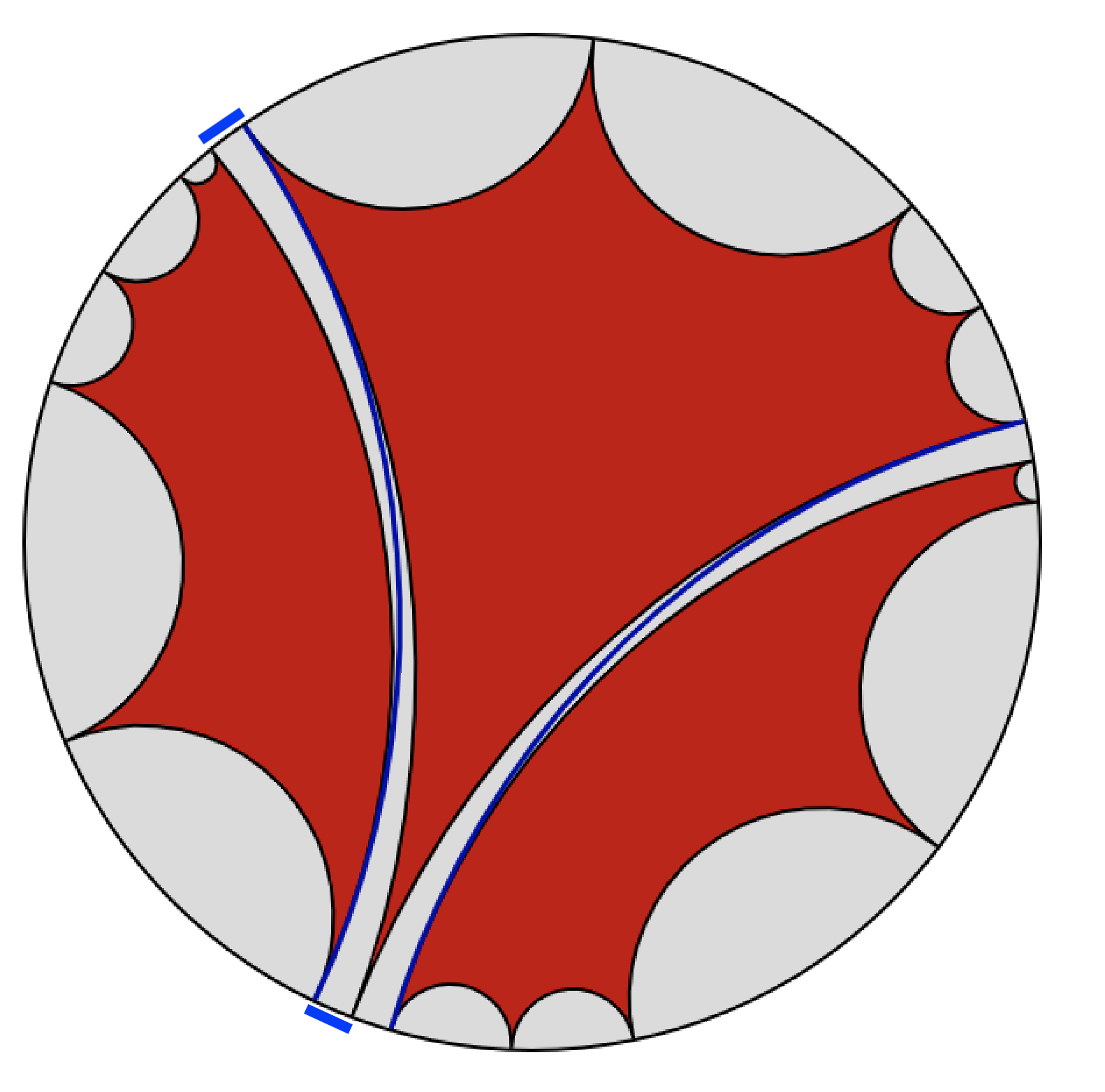}
    \caption{Here we have an initial polygon on the left and its first pullback on the right. On the right, the two endcaps associated with the left sibling are highlighted.}
    \label{fig:endcap}
\end{figure}

\subsection{Central Strips}
\begin{defn}[Central Strip]
Consider the {\em sibling portrait} of a full collection, $\mathcal S$, of sibling leaves where no leaf maps to a diameter and at least one leaf is longer than $\frac{1}{d+1}$. The {\em central strip} is the union of components of $\cl\disk \setminus \cup \mathcal{S}$ such that these components contain at least two arcs of the circle of length less than $\frac{1}{2d}$.
\end{defn}

\begin{thm}[Central Strip Lemma  \cite{Cosper:2016}]\label{thm: CSL}
Let $C$ be a central strip of a leaf, $\ell$, and its siblings where the length, $\eta$, of any arc of $C$ is less than $\frac{1}{d(d+1)}$ (This is a {\em narrow} central strip). Then, the following hold:

\begin{enumerate}
\item The first image $\ell_1=\sigma_d(\ell)$ cannot re-enter $C$.
\item The second image $\ell_2=\sigma_d^2(\ell)$ cannot re-enter $C$ with both endpoints in a single component of $C\cap\ucirc$.
\item If an iterate
$\ell_j=\sigma_d^j(\ell)$ of $\ell$ re-enters $C$, for least $j>2$,
and has endpoints lying in one component of $C\cap\ucirc$,
then iterate $\ell_k$, for some $k\le j-1$, gets at least as close in the
endpoint metric as
 $\displaystyle\frac{\eta}{d^{j-k}}$ to a critical chord in $\cl\disk\setminus C$.
 \end{enumerate}
\end{thm}

The proof of the following corollary, first  stated in \cite{Cosper:2016}, is left to the reader.

\begin{cor}[Unicritical Central Strip Lemma]\label{cor: uni CSL}
 Let $C$ be the central strip of leaf, $\ell$, and its siblings for the map $\sigma_d$ in a unicritical lamination.  Then, no image of $\ell$ can re-enter $C$ with both endpoints in a single component of $C\cap\ucirc$.
\end{cor}

\begin{cor}[Maximal Central Strip Lemma]\label{max cen strip}

Let $P$ be a maximally critical SCM polygon. Then, the orbit of any major cannot re-enter $C$ with both endpoints in a single component of $C\cap\ucirc$.

\end{cor}
\begin{proof}
Due to the Central Strip Lemma (Theorem \ref{thm: CSL}), it is sufficient to show that a major leaf must be within $\dfrac{1}{d(d+1)}$ of critical length. Since we have an SCM polygon, the only critical length is $\dfrac{1}{d}$. 

First, let us consider the identity return case. Our SCM polygon consists of $d-1$ majors of length greater than $\dfrac{1}{d}$. There is only room for one other leaf in a disjoint orbit from the $d-1$ majors. Call this leaf, $M$. By our definition of an SCM polygon (Definition \ref{def:SCM}), $|M| > \dfrac{1}{d+1}$.

By way of contradiction, assume there exists a leaf, $\ell$, such that $\ell$ is not within $\dfrac{1}{d(d+1)}$ of critical length. This implies $|\ell| > \dfrac{1}{d} + \dfrac{1}{d(d+1)} = \dfrac{d+2}{d(d+1)}$. The length of the other majors is $|\sum_{i=1}^{d-2}(\ell_i)| > \dfrac{d-2}{d}$. Thus, the length of our circle used up by the majors is $|\ell| + |\sum_{i=1}^{d-2}(\ell_i)| > \dfrac{d+2}{d(d+1)} + \dfrac{d-2}{d} = \dfrac{d}{d+1}$. This leaves less than $\dfrac{1}{d+1}$ for $M$, a contradiction. 

Now, let us consider the rotational case. Our SCM polygon consists of $d-1$ majors of length greater than $\dfrac{1}{d}$ and all of the pre-images of those majors.

Again, assume by way of contradiction that there exists a leaf, $\ell$, such that $\ell$ is not within $\dfrac{1}{d(d+1)}$ of critical length. As before, the majors take up more than $\dfrac{d}{d+1}$ of the circle. This leaves $\dfrac{1}{d+1}$ for all the pre-image leaves. If we consider the immediate pre-image of $\ell$, we can simply divide by $d$ to find its length since we are one-to-one except at the major. Thus, we have that $|\sigma_d^{-1}(\ell)| = \dfrac{|\ell|}{d} > \dfrac{d+2}{d^2(d+1)}$. Similarly, we can find the length of the immediate pre-images of the other majors: $|\sum_{i=1}^{d-2}\sigma_d^{-1}(\ell_i)| = \dfrac{|\sum_{i=1}^{d-2}(\ell_i)|}{d} > \dfrac{d-2}{d^2}$. Thus, the length of the pre-images of the majors is greater than  $\dfrac{d+2}{d^2(d+1)} + \dfrac{d-2}{d^2} = \dfrac{1}{d+1}$. However, we only have less than $\dfrac{1}{d+1}$ left from the majors, a contradiction.
\end{proof}



    

\subsection{Topological Julia Set}

\begin{defn}[First Return Map]
Starting with a $d$-invariant lamination and a Fatou gap, $G$, the {\em first return map}, $R$, of $G$ is the first iterate of $\sigma_d$ where $G$ returns to itself.
\end{defn}

\begin{rem}
The idea that $G$ returns to itself after some $\sigma^k_d$ makes sense, because we know $G$ is periodic.
\end{rem}

\begin{defn}[Topological Julia Set]
Starting with a $d$-invariant lamination, $\mathcal{L}$, let $\sim_\mathcal{L}$ be the induced equivalence relation. Let $J(\mathcal{L})$ be the quotient space of $\sim_\mathcal{L}$ and let $h : \ucirc \to J(\mathcal{L})$ be the natural projection. We call $J(\mathcal{L})$, a {\em topological Julia set}.   
\end{defn}

\begin{rem}
We can think of the quotient map as going from the lamination to $J(\mathcal{L})$ shrinking leaves and polygons to points. It follows from the definition of $d$-invariant lamination that equivalence classes map to equivalence classes, and the pre-image of equivalence classes is a union of equivalence classes.
\end{rem}

\begin{defn}[Induced Return Map]
We call $$P(x) = h(\sigma(h^{-1}(x))) : J(\mathcal{L}) \to J(\mathcal{L})$$ a {\em topological polynomial}. If we take $\hat{G} = h(\partial G)$ then $R$ induces a map $\hat{R} : \hat{G} \to \hat{G}$ which we call the {\em induced return map}.
\end{defn}

Where no confusion will result, we will use $R$ for the first return map on both laminations and topological Julia sets.

Note a map $f: \ucirc \to \ucirc$ is {\em topologically exact} iff for any interval $I \subset \ucirc$, there exists some $n$ such that $f^n(I) = \ucirc$. The following theorem is well known. (See \cite{Blokh:2006}.)

\begin{thm}
A topologically exact covering map $f: \ucirc \to \ucirc$ is conjugate to $\sigma_d$ for some $d \ge 2$.
\end{thm}

\begin{defn}[Hyperbolic Lamination]
A $d$-invariant lamination is said to be \emph{hyperbolic} if and only if all compatible critical chords are interior to periodic Fatou gaps. 
\end{defn}

\begin{prop}
Canonical MAC laminations and canonical SCM laminations are hyperbolic.
\end{prop}

\subsection{Structure of Canonical Unicritical Laminations}

\begin{lem}[Fatou Gaps of Unicritical Laminations \cite{Schleicher:1999}]\label{Top1}
Let $\mathcal{L}$ be a degree $d$ unicritical lamination with a periodic major, $M$. The following hold: 
\begin{enumerate}
    \item There is a periodic Fatou gap, $G$, compatible with an all critical $d$-gon. Hence, $\mathcal{L}$ is hyperbolic. 
    \item All boundary leaves of $G$ are pre-images of the orbit of $M$, and the convex hull of $G\cap \ucirc$ is $G$.
    \item The induced return map $\hat{R}: \hat{G} \to \hat{G}$ is conjugate to $\sigma_d$.
\end{enumerate}
\end{lem}

\begin{prop}[MAC Orbit Location] \label{prop: MAC orbit Location}
Let $\mathcal{L}$ be a unicritical lamination with a MAC leaf, $M$. Then, all of the leaves in the MAC orbit lie outside the central gap bounded by the major and its siblings. 
\end{prop}

\begin{proof}
 We have a unicritical central strip with small arcs of the $d$-gon between the major's siblings and the all critical $d$-gon by Definition \ref{def: MAC}. By Corollary \ref{cor: uni CSL}, this central strip cannot be entered by any part of the forward orbit of the major and a leaf can not have both endpoints in one end cap. If we connected two different endcaps with a leaf, then such a leaf would be closer to critical length than our initial major which contradicts our choice of major leaf.
\end{proof}

\begin{thm}
In a MAC lamination with a rotational polygon, there are no limit leaves and the pullback leaves limit to points on the circle. 
\end{thm}

\begin{proof}

Let $M$ be our major leaf in a MAC lamination. In a MAC lamination, we have a symmetric sibling portrait so every critical sector is of length $\dfrac{1}{d}$. A single critical sector maps one-to-one onto the circle. Thus, as we do each step of the pullback lamination (Definition \ref{sec: pullback}), there is a pre-image in every critical sector. As we pullback again, the critical sector is broken into $d$ equal intervals each again with a pre-image in it. As we continue the pullback process, our intervals will continually get smaller and cause our pullback leaves to limit to points on the circle.  

\end{proof}


The following definition and theorem are adapted from \cite{Schleicher:1999}.

\begin{defn}[Co-root]
A {\em co-root} is a point, other than an endpoint of the major, in the boundary of the central gap of the MAC lamination that is fixed under the first return map.
\end{defn}
\begin{rem}
Each co-root is in a different endcap of the siblings of the MAC leaf. See Theorem \ref{thm: co-root}.
\end{rem}

\begin{thm}[Co-root Theorem]\label{thm: co-root}
Given a degree $d$ MAC lamination, there will be $d-2$ co-roots in the end caps of the central strip that are not adjacent to the major. The distance between co-roots is greater than $1/d$.
\end{thm}

\begin{proof}
Let $G$ be the degree $d$ Fatou gap of a MAC lamination with major, $M$, and return map, $R$. Applying the quotient map $h$ to $\partial G$ takes the leaves bounding our Fatou gap to points as pictured in Figure \ref{Co-root Diagram}. Thus, $h$ takes $\partial G$ to a topological circle $\hat{G}$ with a first return map $\hat{R}$ conjugate to $\sigma_d$ (Lemma \ref{Top1}). Since $\hat{R}$ is of degree $d$, it will have $d-1$ fixed points; the first of which we can identify as $h(M)$. The rest of the fixed points under $\sigma_d$ must be equally spaced around the circle with spacing $\frac{1}{d-1}$ and are between the siblings of $M$. We claim that these points under $h^{-1}$ are points in $\partial G$ and not leaves. All of the points in $\hat{G}$ corresponding to leaves in $\partial G$ eventually map to the point $h(M)$ because all of the leaves in $\partial G$ eventually map to $M$ (Lemma \ref{Top1}). Points in $\hat{G}$ corresponding to leaves cannot be fixed under $\hat{R}$, otherwise they would not map to $h(M)$. Now, we apply $h^{-1}$ to each of these points to get points in $\partial G$ between the siblings of $M$ which are our co-roots.
\end{proof}

The following algorithm gives rise to the existence of co-roots for all unicritical laminations. 

\begin{rem}[The Generalized Lavaur's Algorithm, Section 6 of \cite{Bhattacharya:2021}]\label{thm: lavaurs} 
For all unicritical laminations of any degree, we can uniquely find the corresponding minor leaf that generates the lamination. In this process of finding these minors, the algorithm \enquote{skips} over certain points of every period. These points that are skipped over are the desired co-roots. However, the Lavaur's Algorithm identifies the minor leaf rather than the major. Thus, the \enquote{co-roots} it identifies are actually the images of the points described in the co-root definition. These points are in the same orbit so we will use the term interchangeably when no confusion will arise. We will identity a co-root with the leaf it is associated with in its orbit. 
\end{rem}

\begin{figure}
    \centering
    \includegraphics[width = 5in]{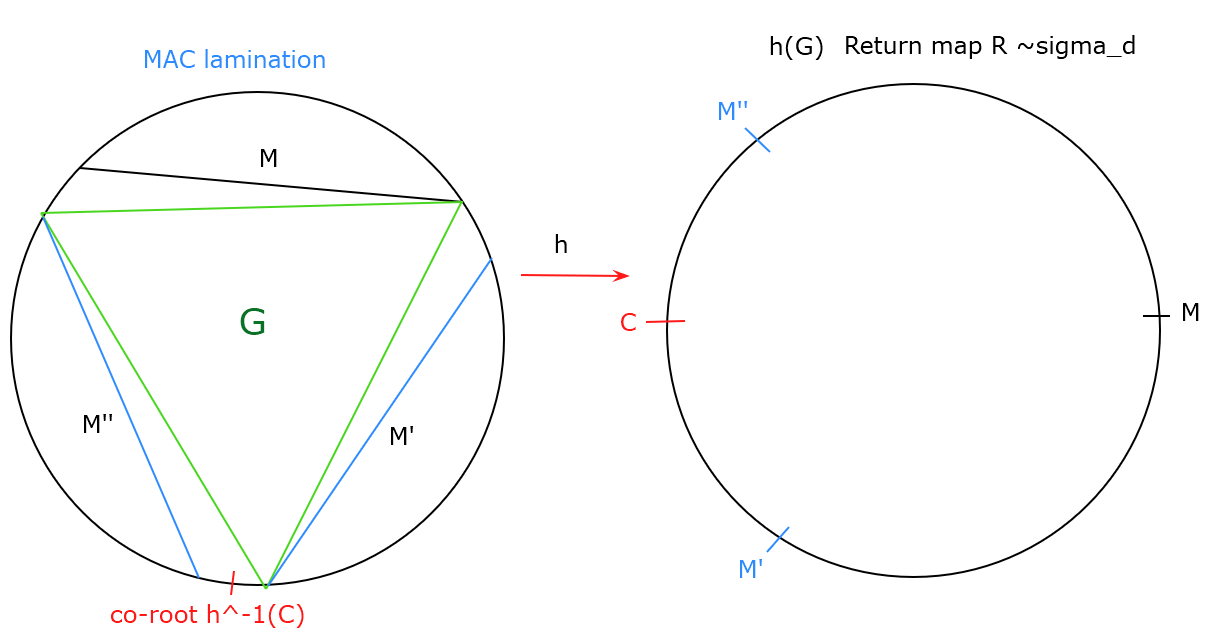}
    \caption[Co-root Diagram]{Here is a MAC leaf in a lamination bounding a Fatou gap $G$ (Left). The quotient map, $h$, applied to $G$ gives $h(G)$ (Right). In $h(G)$, the point labeled $M$ is technically $h(M)$. The return map on $h(G)$ is conjugate to $\sigma_d$.}
    \label{Co-root Diagram}
\end{figure}

\begin{lem}\label{lem: disjoint gaps} 
The forward images of the Fatou gap, $G$, in a MAC lamination are pairwise disjoint except possibly at the endpoint of $M$ or an image of $M$.

\end{lem}

\begin{proof}
Given a MAC lamination with major, $M$, and Fatou Gap, $G$, along with their respective forward images, want to show the forward images of $G$ are pairwise disjoint except possibly at endpoints of $M$ or its images.

First, we need to show at no point in the forward orbit of $G$ that a gap has two leaves of the MAC orbit in its boundary.  By way of contradiction, assume at some iterate $k$ in the forward orbit of $G$ there were two leaves of the MAC orbit in the boundary of $G_k$. The forward image, $G_{k+1}$, of $G_k$ will also have two leaves of the MAC orbit in its boundary since the gap will map forward one to one until it returns to $G$. When we return to $G$, it will have two leaves in the MAC orbit in its boundary, one of which being $M$. The other leaf, $\ell$, maps into the central strip formed by $M$ and its siblings. The leaf, $\ell$, cannot enter the central strip long because it will be closer to critical than $M$ which contradicts the fact that $M$ was our major leaf. Thus, $\ell$ must enter the central strip short which implies that both endpoints of the leaf are in the same component of the central strip. However, this contradicts Corollary \ref{cor: uni CSL}. Therefore, the forward orbit of $G$ only has one leaf in the MAC orbit in its boundary which is $M$ or an image of $M$ depending on the iterate of $G$.

Second, we need to show that gaps in the forward orbit of $G$ cannot meet at a leaf. By way of contradiction, assume that $G_i$ and $G_k$ with $i<k<n$ meet a leaf where $n$ is the period of $M$. All of the leaves in the boundary of $G$ and its images eventually map into the MAC orbit (Lemma \ref{Top1}). The leaf shared between the two iterates, $G_i$ and $G_k$, will eventually become $M$. Then, $G_k$ will map to $G$ in $n-k$ steps, and $\sigma^{n-k}(G_i)$ will only be attached to $G$ by $M$. Then, $\sigma^{n-k}(G_i)$ will map to $G$ in $k$ steps before $G = \sigma^{n-k}(G_k)$ will return to $G$ since $k<n$. Since $M$ returns to $G$ in fewer than $n$ steps, a forward image of $M$ must be in $G$. However, this contradicts the results of the previous paragraph.

In the rotational case, a forward image of $G$ will meet $G$ at an endpoint of $M$ since the rotational MAC leaf will touch two of its forward images. 
\end{proof}

\begin{cor} 
In the identity return case, the forward images of $G$ are always pairwise disjoint.
\end{cor}

\begin{proof}
We will verify that no two gaps in the forward orbit of $G$ share a point. Note that by Proposition \ref{prop: gap meet}, we know two gaps cannot meet at a point that is not an endpoint of leaves. Suppose by way of contradiction, there is a boundary leaf in $G_i$ that shares an endpoint with a boundary leaf in $G_k$. Then, we have two leaves that will eventually map into the MAC orbit attached at a point. The forward images of these two leaves must also be attached. However, an identity return MAC leaf and its images must be disjoint, thus a contradiction.
\end{proof}

We currently have enough information to prove one direction (MAC to SCM) of this one-to-one correspondence. We want to show, given any degree $d$ MAC lamination, there is a unique (canonical) SCM lamination to which it corresponds. 

\begin{figure}
    \centering
    \includegraphics[width = 1.9in]{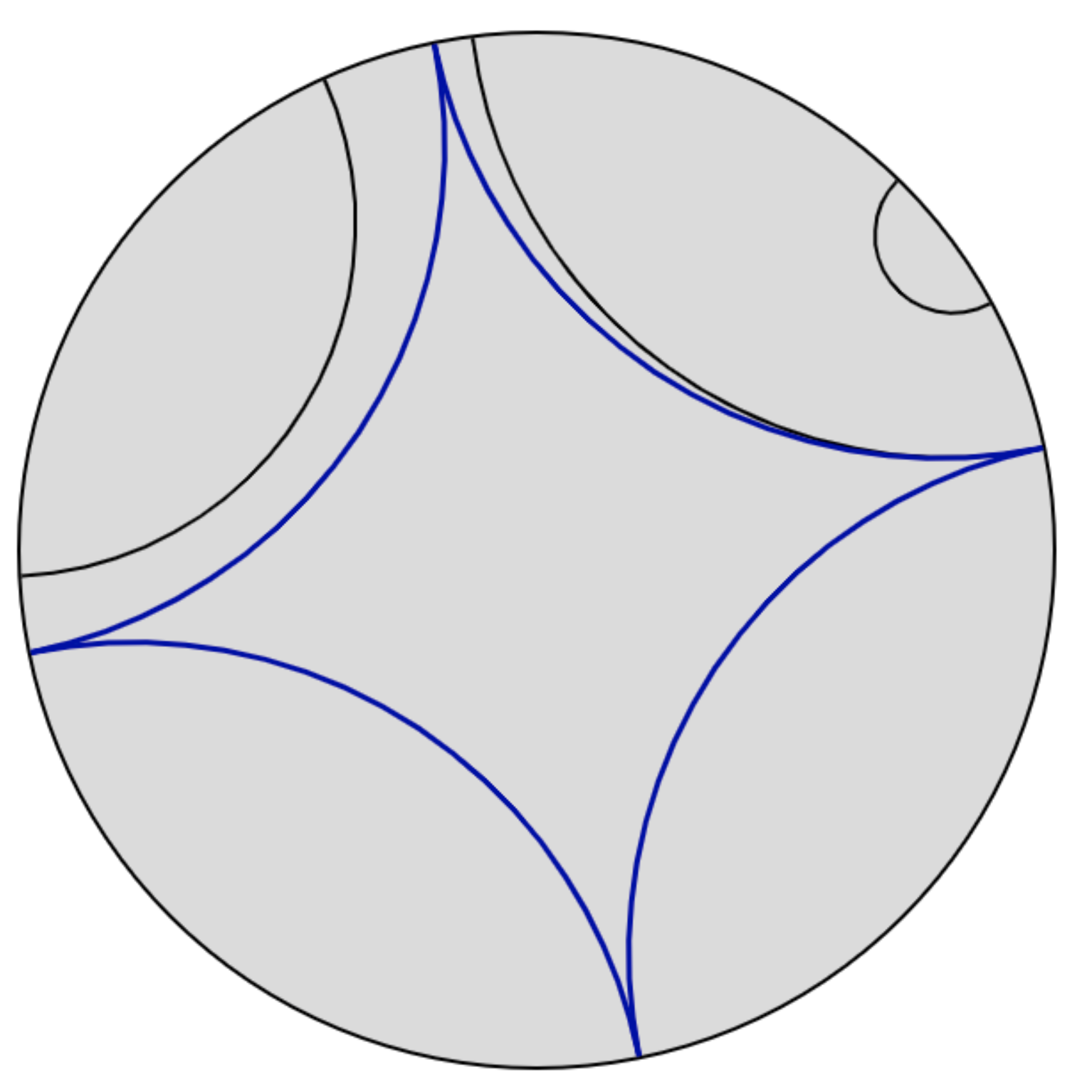}
    \includegraphics[width = 1.9in]{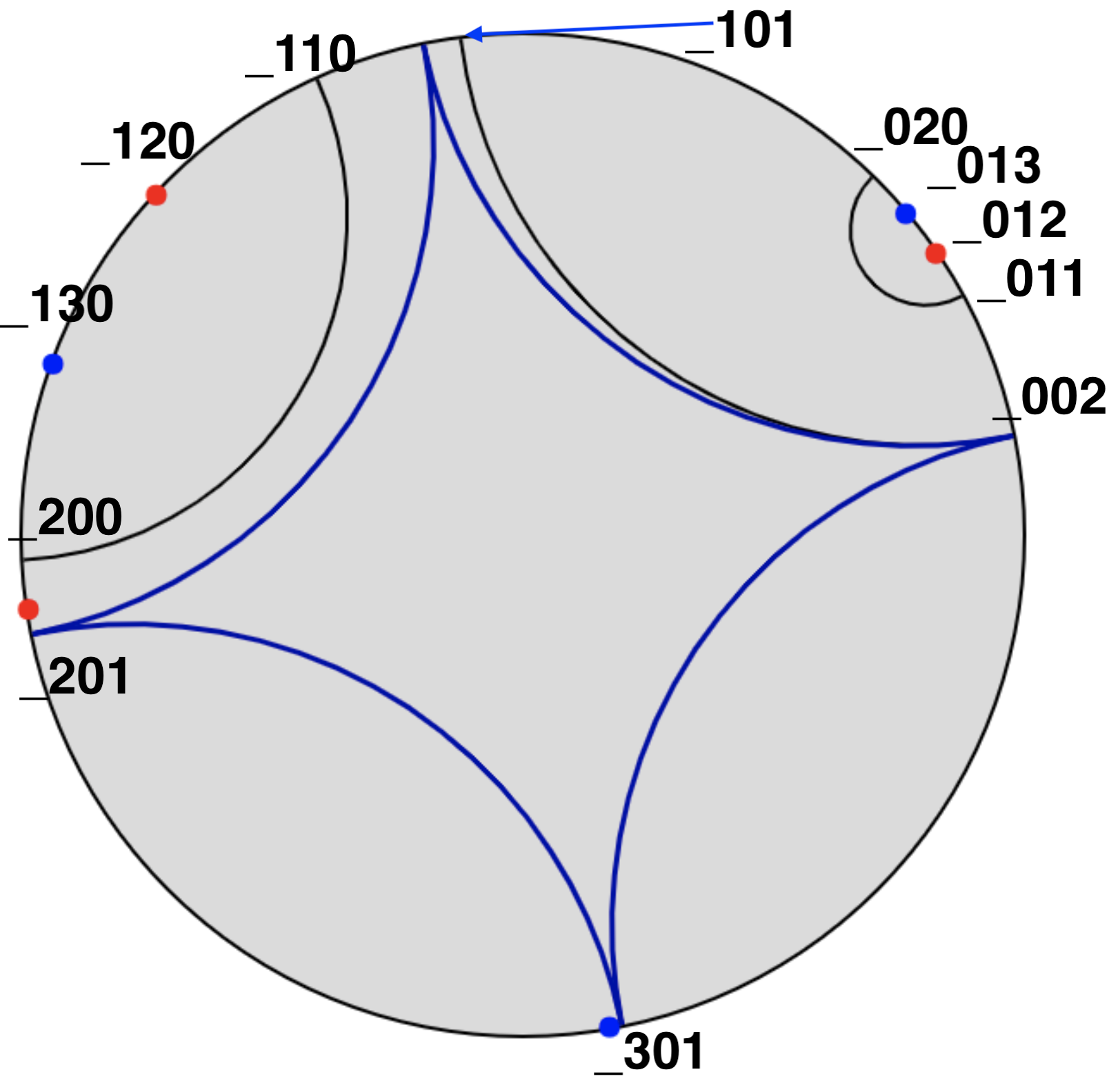}
    \includegraphics[width = 1.9in]{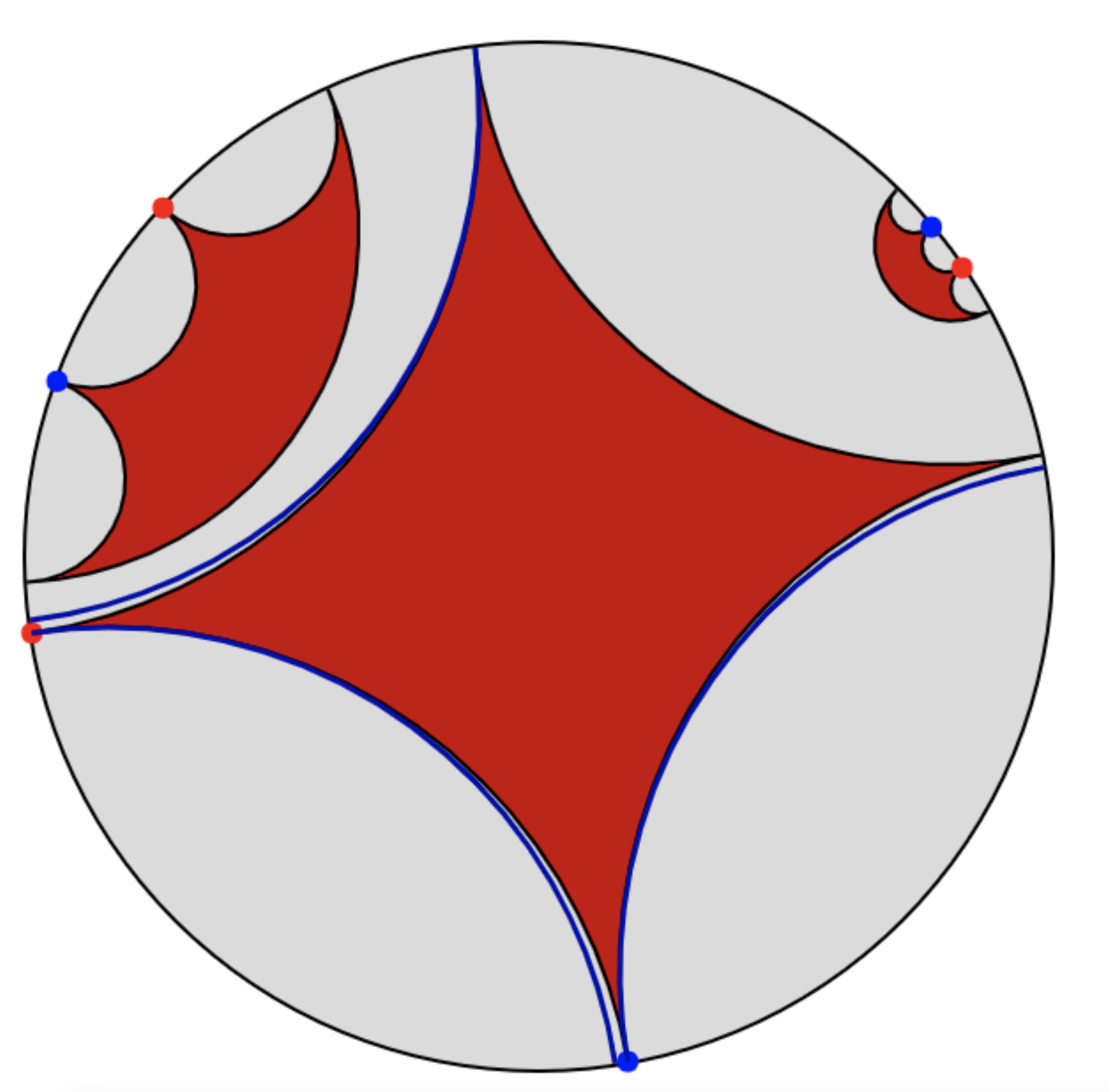}
    \caption[Identity Return Case]{On the left, we have a degree 4 MAC identity return leaf and its full forward orbit. In the middle, we have the co-roots and their images added. On the right, we have the resulting SCM polygon.}
    \label{ir mac to scm}
\end{figure}

\begin{figure}
    \centering
    \includegraphics[width = 1.9in]{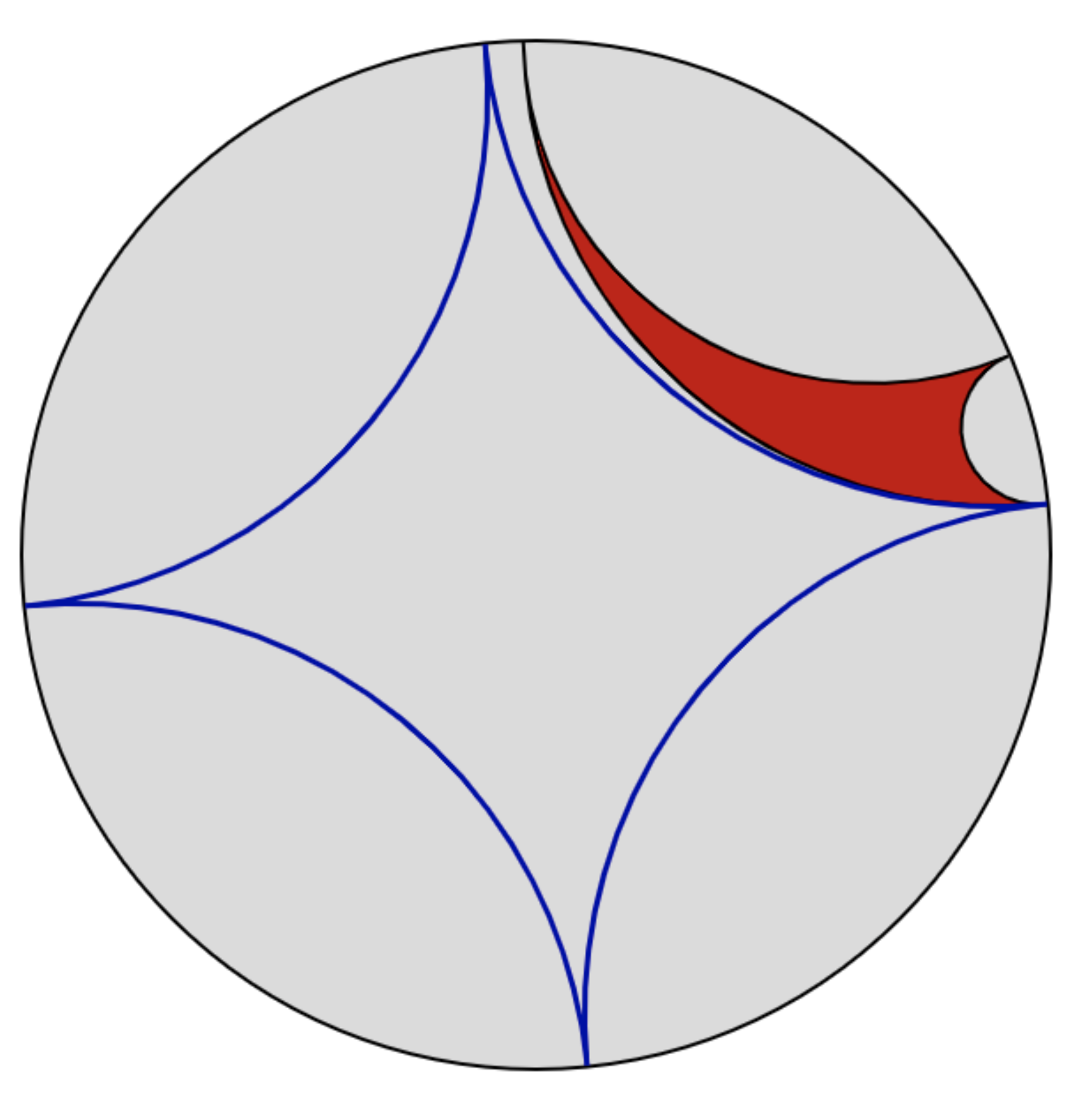}
    \includegraphics[width = 1.9in]{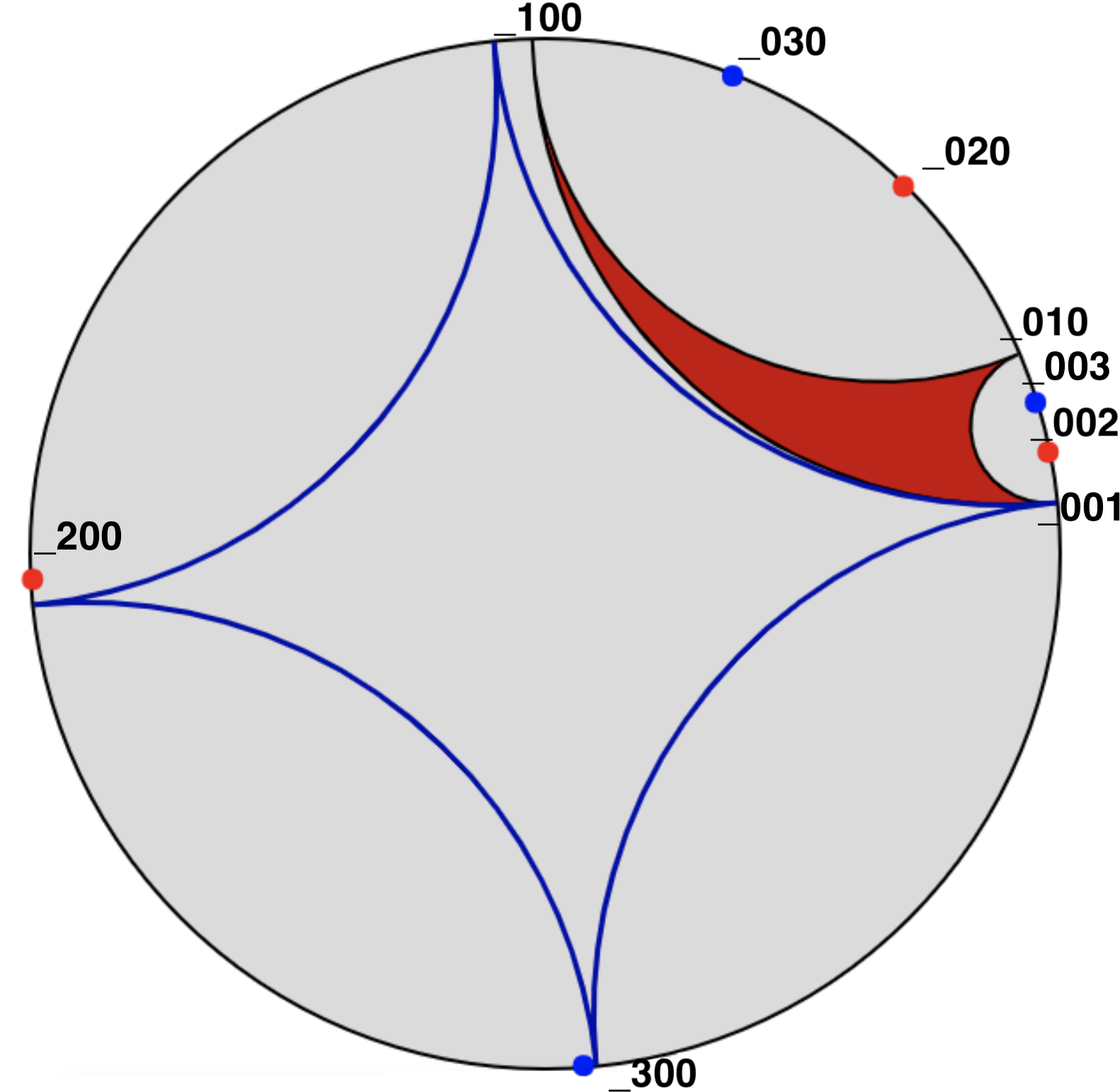}
    \includegraphics[width = 1.9in]{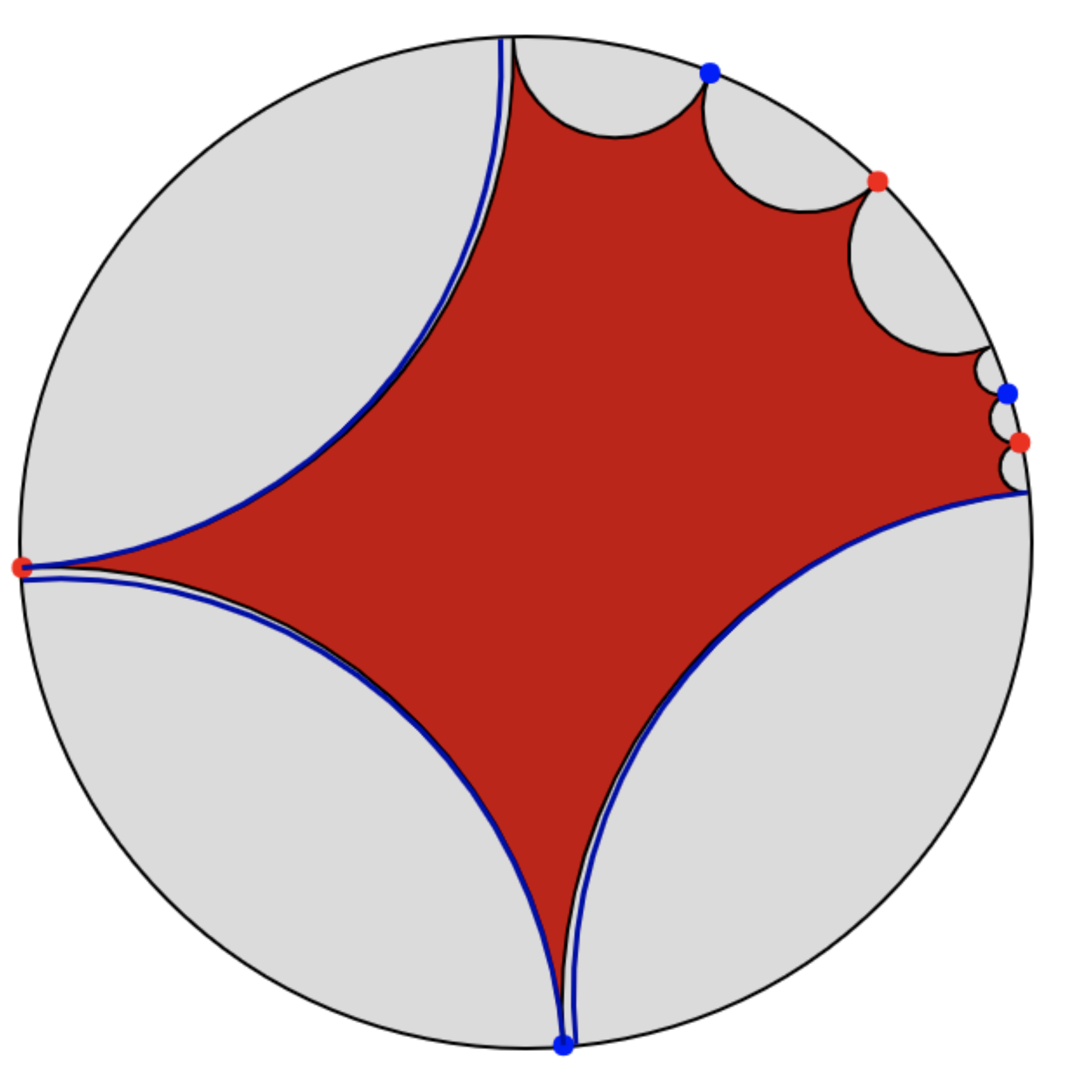}
    \caption[Rotational Case]{On the left, we have a degree 4 MAC rotational polygon. In the middle, we have the co-roots and their images added. On the right, we have the resulting SCM polygon.}
    \label{rot mac to scm}
\end{figure}

\begin{lem}[MAC to SCM] \label{lem: MAC to SCM}
Starting with a MAC leaf, $M$, there is a unique corresponding SCM polygon, $P(M)$, with adjacent majors. In the identity return case, we will have a $d$-gon and a $k(d-1)$-gon in the rotational case.
\end{lem}

\begin{proof}

Let $\mathcal{L}$ be MAC lamination with minor leaf, $m$, and major leaf, $M$. By Theorem~\ref{thm: co-root} we have the following statements. There are $d-2$ co-roots associated with $M$. The co-roots of the major always lie on the long arc of the circle subtended by $M$. The siblings of $M$ are symmetric about the circle lying in different critical sectors. These siblings have endpoints on the long arc of the circle subtended by $M$. The siblings are between either an endpoint of the major and a co-root or between two co-roots.

Let us first discuss the identity return case. We can construct a $d$-gon, $P(M)$, using $M$ and its co-roots. We will remove the current critical chords and justify the existence of a new complete set of compatible $d-1$ critical chords. We consider the short arc subtended by $M$. Along that arc, one endpoint is counterclockwise of the other. We chose that as our starting point (see Figure \ref{cw endpoint of M}). Now, take this point and connect it to the nearest co-root in the counterclockwise direction with a leaf. Doing so gives us two connected leaves, or two sides of our $d$-gon, $P(M)$, we want to construct. There are now $d-3$ co-roots remaining and the other endpoint of the major to connect together. Proceeding in the counterclockwise direction, we connect the rest of the points in similar fashion until we arrive at the other endpoint of $M$ giving us an additional $d-2$ sides. This gives a total of $d$ sides making a $d$-gon with $M$ as a side and the co-roots as vertices. 

The new leaves connecting the endpoints and co-roots are longer than the siblings of the major; this follows from the fact that they're closer to critical because they're in the end cap. By Theorem~\ref{thm: co-root} we know that the distance between co-roots is greater than $\frac{1}{d}$. Since these new sides are all longer than critical length, we know that there is room for a critical chord under $d-1$ of the sides of our new $d$-gon, $P(M)$, putting it into a maximal critical sector. 

We will now show that $P(M)$ is SCM with the same period as $M$. $M$ and its co-roots have the same period since they are fixed under the first return map of the gap, $G$, for which $M$ is the major (Lemma \ref{Top1}). So, all of the vertices of $P(M)$ have the same period. Since $P(M)$ lies entirely in $G$, and $G$ maps forward preserving circular order, the vertices, and thus, the sides of $P(M)$ map forward in order.

Since $M$ is identity return and a side of $P(M)$, then $P(M)$ must also be identity return. To verify that $P(M)$ is the closest to critical in its orbit we look at its side lengths compared to the side lengths of its forward images. Since $P(M)$ is contained in $G$ it will map with the forward orbit of $G$. This forward orbit follows the forward orbit of $M$ so that $G$ maps immediately under the image of $M$ and stays on the short arc of the images of $M$ until returning to $M$.

Note that the shortest side of $P(M)$ is $M$ because the other leaves in $P(M)$ are closer to critical that $M$. 

Thus, when mapping forward under $\sigma_d$, $M$ will be the longest side of the forward images of $P(M)$ until returning to $P(M)$. The sides of the images of $P(M)$ are bounded under the corresponding images of $M$. The sides of $P(M)$ will not be able to approach criticality until $M$ does, and this is the iterate with $M$ as a side. Therefore $P(M)$ is the closest to critical in its orbit. Now we have $P(M)$ is identity return, in a maximal critical sector, and is the closest to critical in its forward orbit. Thus, $P(M)$ is an SCM $d$-gon by Definition \ref{def:SCM}. The construction guarantees its uniqueness and completes the identity return case.

Now, let us consider the rotational case. Again, we will remove the current critical chords and put in compatible ones with the new polygon. We construct a $k(d-1)$-gon using $M$ and its co-roots where $k$ is the period of the endpoints of the rotation return polygon. We choose the endpoint of $M$ as before. We connect it and the nearest co-root in the counterclockwise direction. This is one leaf of our $k(d-1)$-gon. There are now $d-3$ co-roots remaining and the other endpoint of the major to connect together. Proceeding in the clockwise direction we connect the rest of the points in similar fashion until we arrive at the other endpoint of $M$. This creates $d-1$ new leaves that subtend $M$. We maintain the endpoints of $M$, but the original leaf has been replaced by the leaves generated by the co-roots and the endpoints of $M$.
As we repeat this process for the other $k$ sides of the rotational polygon and the corresponding co-root images, we have a total of $k(d-1)$ sides of our new polygon. The original $k$ sides of rotational polygon have each been replaced by $d-1$ new sides. 

Similarly to the identity return case, we have adjacent majors and our $k(d-1)$-gon is in a maximal critical sector so we have room for $d-1$ compatible critical chords. Our polygon is the same period as $M$ and rotational with the same rotation number as $M$. Thus, $P(M)$ is a rotational SCM $k(d-1)$-gon by Definition \ref{def:SCM}. The construction again guarantees its uniqueness and completes the rotational case.

\begin{figure}
    \centering
    \includegraphics[width = 3in]{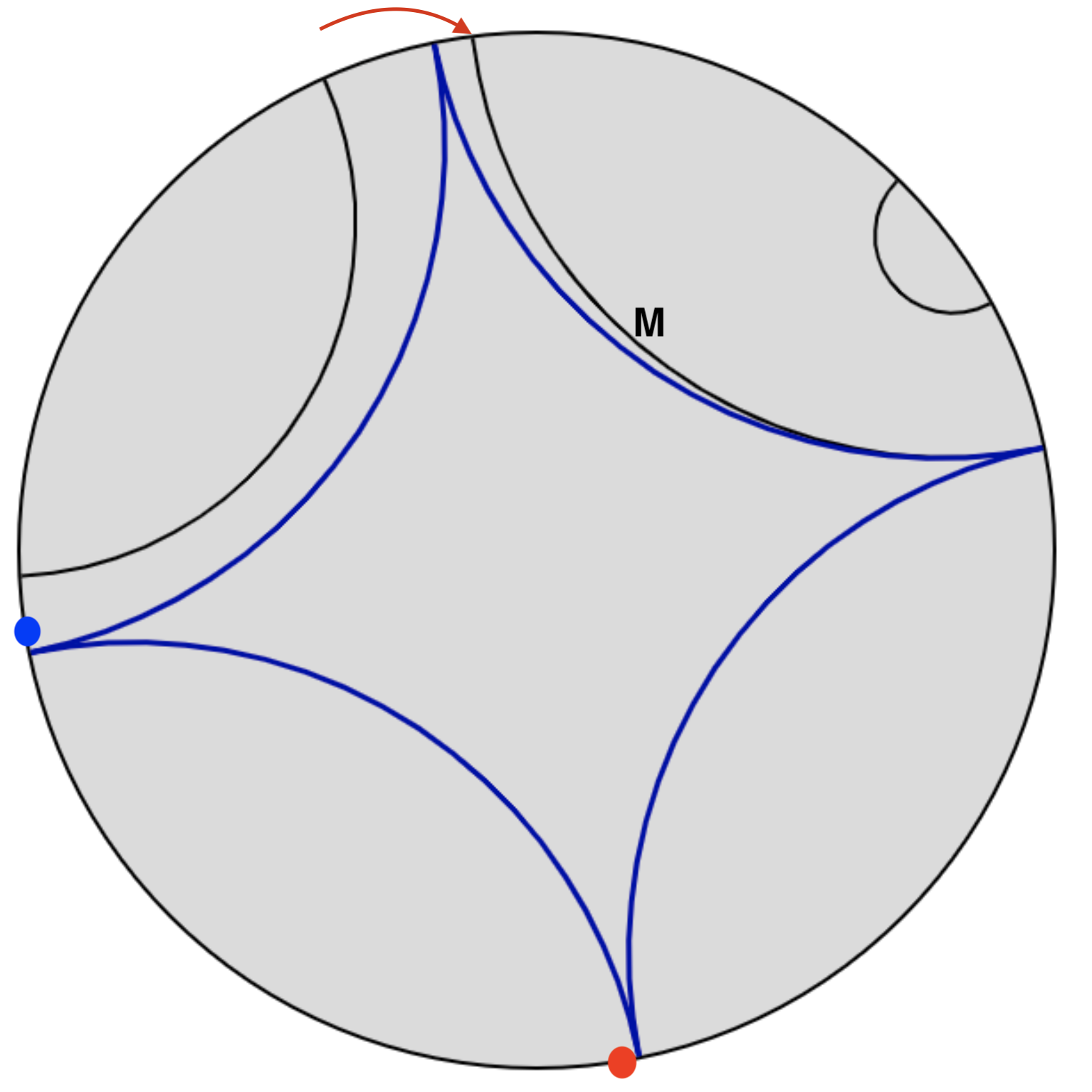}
    \caption{The endpoint of $M$ at the clockwise end of the short arc subtended by $M$ in the identity return case with the two respective co-root points.}
    \label{cw endpoint of M}
\end{figure}


\end{proof}

\begin{figure}
    \centering
    \includegraphics[width = 2in]{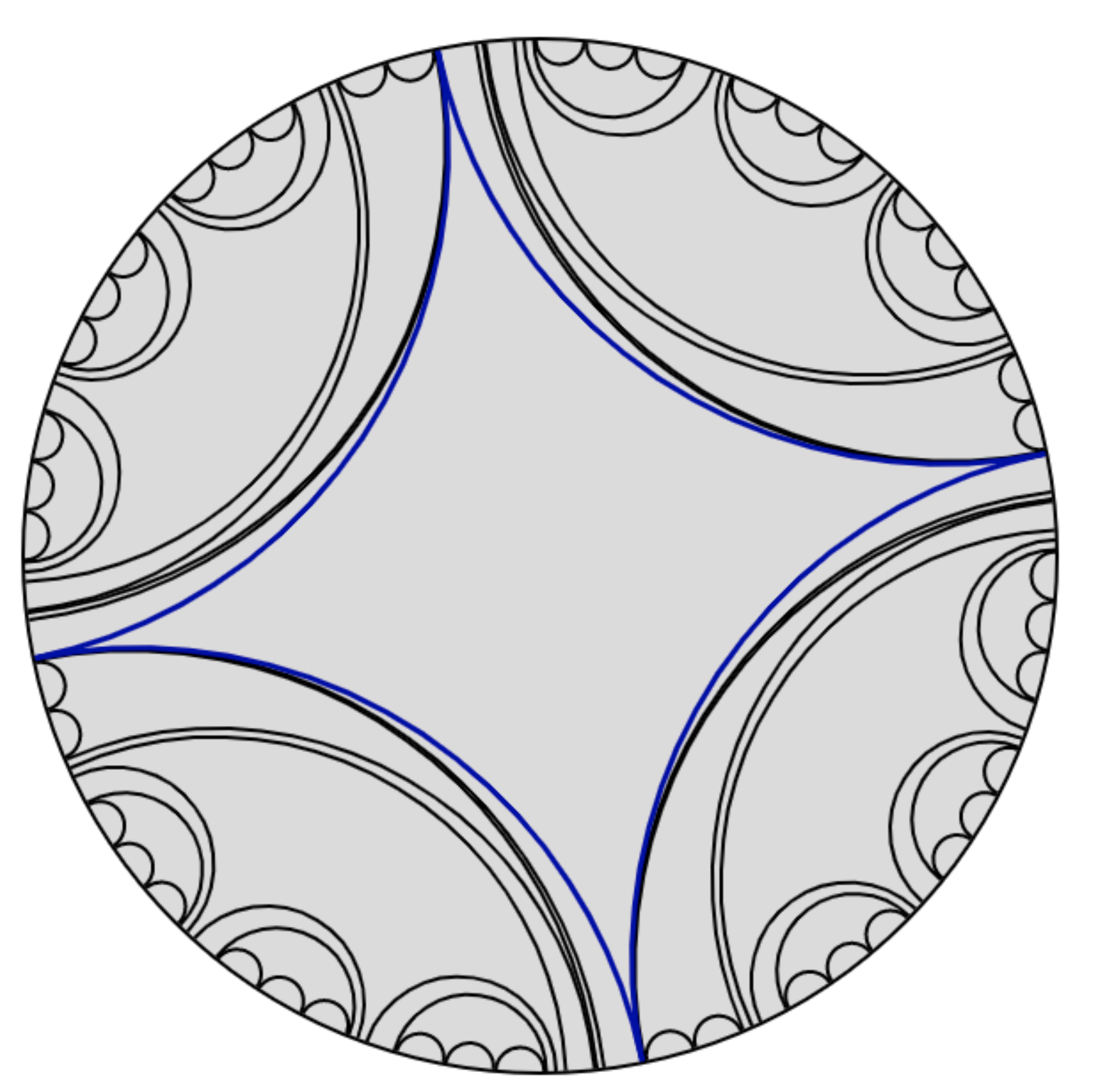}
     \hspace{.3 in}
    \includegraphics[width = 2in]{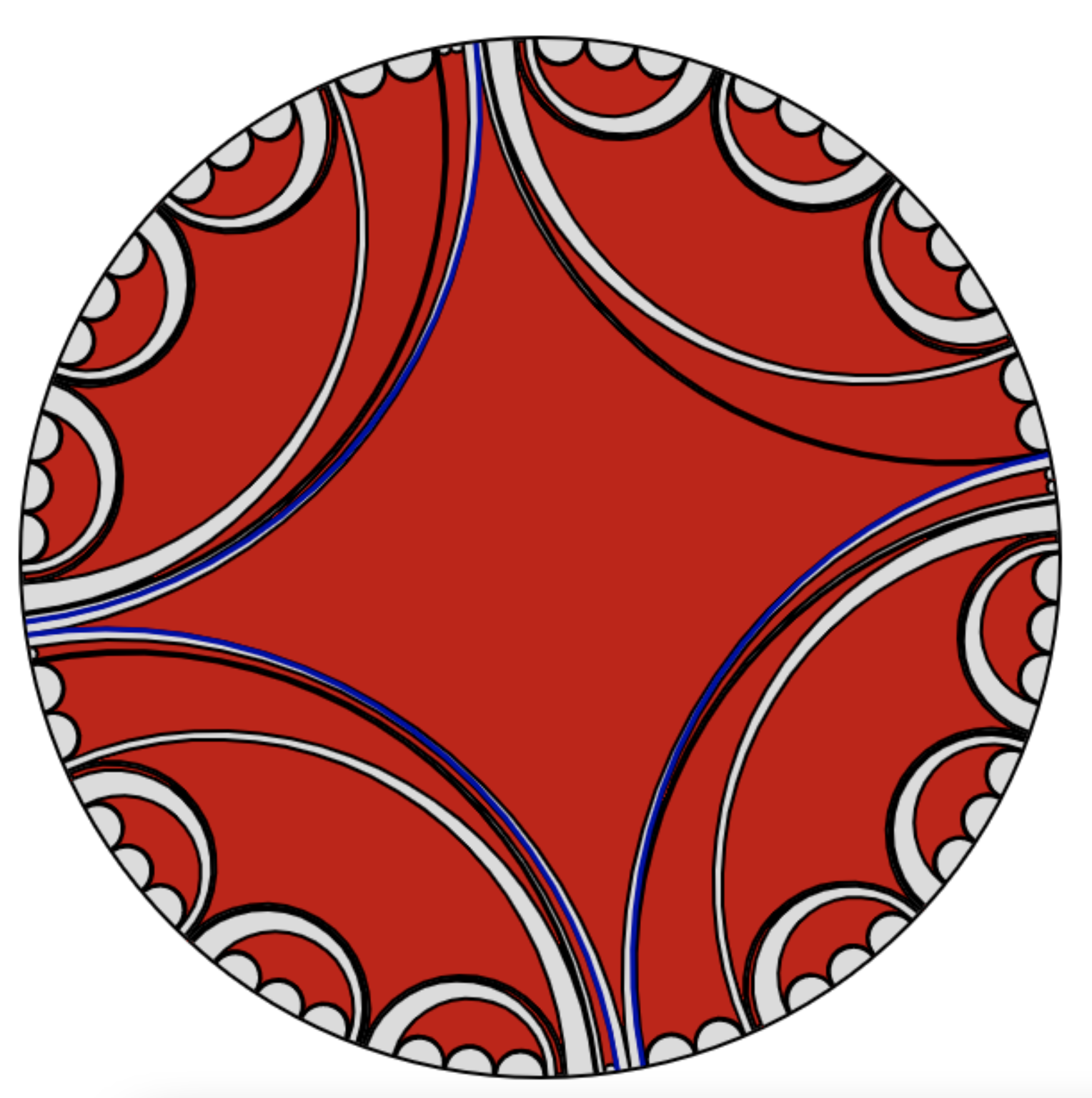}
   
    \caption[Canonical SCM Lamination]{
    Here are the finite pullbacks for Figure \ref{ir mac to scm}. On the left, we have pullbacks of our identity return unicritical leaf, and on the right, we have pullbacks for our SCM identity return polygon.}
    \label{fig:canonical ir SCM}
\end{figure}

\begin{figure}
    \centering
    \includegraphics[width = 2in]{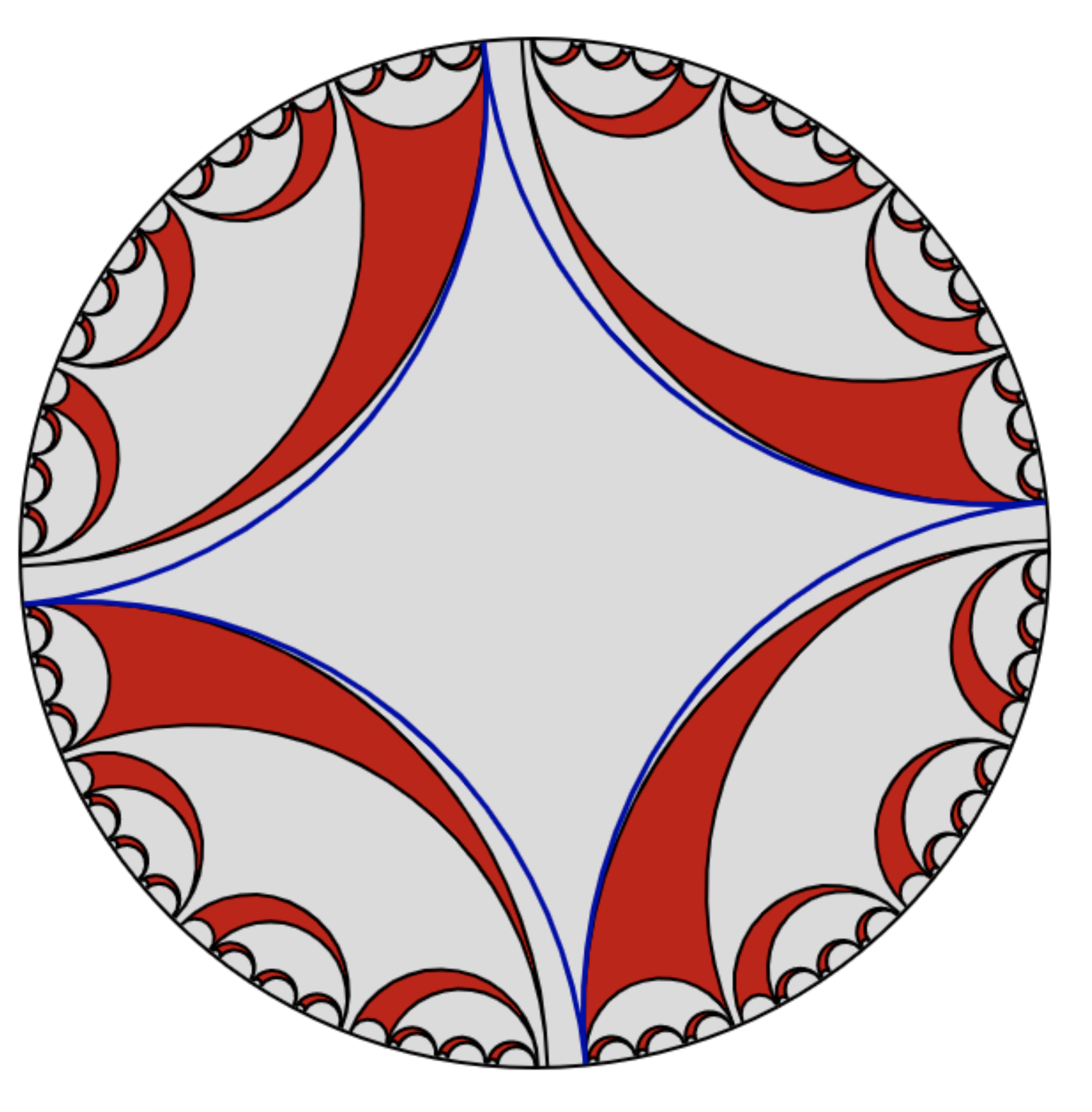}
     \hspace{.3 in}
    \includegraphics[width = 2in]{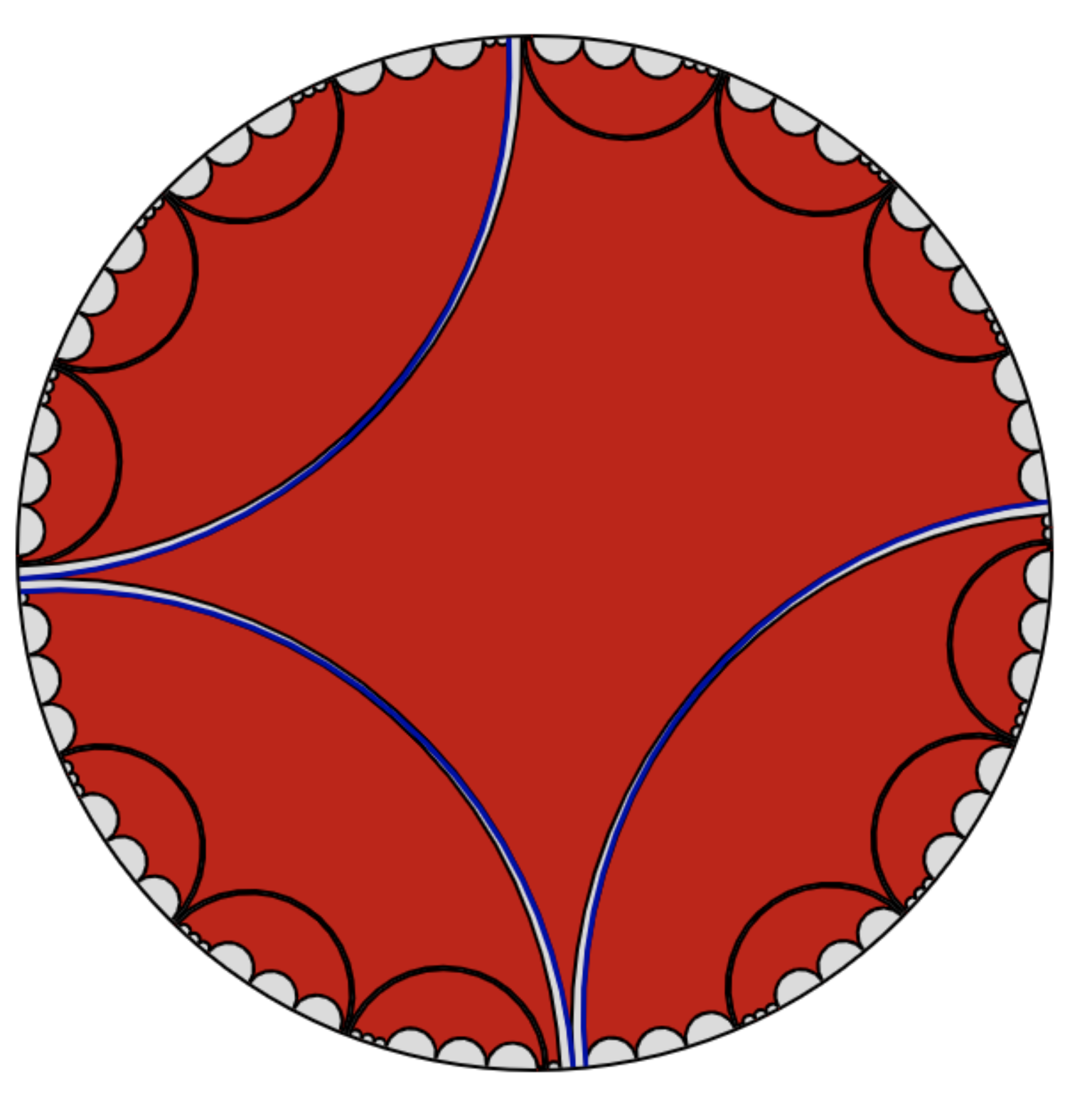}
   
    \caption[Canonical SCM Lamination]{
    Here are the finite pullbacks for Figure \ref{rot mac to scm}. On the left, we have pullbacks of our rotational unicritical polygon, and on the right, we have pullbacks for our SCM rotational polygon.}
    \label{fig:canonical rot SCM}
\end{figure}

\begin{thm}[MAC lamination to SCM lamination] \label{thm: MAC to SCM Lamination}
Let $\mathcal{L}(M)$ be a MAC lamination with the MAC leaf $M$. There is a canonical SCM lamination $\mathcal{S}(M)$ that contains an SCM polygon $P(M)$ as follows:

\begin{enumerate}
    \item Identity Return: SCM $d$-gon with $M$ as one of its sides.
    
    \item Rotational: SCM $k(d-1)$-gon where $M$ is a chord interior to the polygon.

\end{enumerate}

\end{thm}

\begin{proof}
Let $\mathcal{L}(M)$ be a MAC lamination with the MAC leaf $M$. We apply Lemma \ref{lem: MAC to SCM} to obtain the respective SCM polygon $P(M)$, and in the identity return case, we have $M$ as one of its sides. In the pullback lamination of $P(M)$ with guiding critical chords described in Definition \ref{def: SCM lam}, the canonical lamination $\mathcal{S}(M)$ is uniquely determined.
\end{proof}

\begin{thm}[Identity Return MAC Gaps] \label{thm: MAC Gaps}

In the identity return case, the only gap attached to the major of a MAC orbit is the infinite gap containing the all critical $d$-gon. In other words the major and its siblings are single sided limit leaves. 
\end{thm}

\begin{proof}
    Start with a MAC lamination, $\mathcal{L}$, with major, $M$, and central gap, $G$. Since $\mathcal{L}$ is a MAC lamination, all the criticality is used up by the all critical $d$-gon in $G$. The major along with its siblings are the longest leaves in the lamination. We can determine a few things about the side of $M$ opposite $G$. There is no room outside $G$ for a chord, or leaf, of critical length. This means there is no room for a collapsing polygon anywhere since it requires a critical chord in its interior. Also, infinite gaps such as Siegel gaps and Fatou gaps (degree $k$-covering gaps) all require a critical length in their interior or on their boundary. These gaps cannot exist in $\mathcal{L}$ since all the criticality is used up by $G$. 
    
    On the other side of $M$, there could either be a finite gap or an infinite gap. We show each is impossible.
    
    The only infinite gap case left to consider is a pre-image of $G$. Suppose there was such a gap $G_{-k}$, a $k^{\text{th}}$ pre-image of G. Then, there is a pre-image of $M$ on the boundary of $G_{-k}$ that is not a leaf in the MAC orbit or else $G_{-k}$ would be in the forward orbit of $G$. The gap $G_{-k}$ cannot reside in the forward orbit of $G$ or else $G$ would share a side with one of its forward images which is prohibited by Lemma \ref{lem: disjoint gaps}. Then let us consider when $G_{-k}$ maps into the forward orbit of $G$. As $G_{-k}$ maps around, it maps one-to-one until the iterate where the side in the MAC orbit maps back to $M$. When $G_{-k}$ maps onto $G$ the points in the boundary of $G_{-k}$ map forward preserving circular order, so $M$ must flip. This contradicts that $M$ returns by the identity.
    
    There remain three finite gap cases to consider: eventually collapsing polygons, eventually periodic polygons, and periodic polygons. First, we cannot have an eventually collapsing polygon because the only gap containing criticality is $G$ which does not collapse. Second, we consider a pre-periodic polygon, but since this polygon would be attached to the already periodic leaf, $M$, the polygon must have the same period. Therefore, the leaf could not be pre-periodic, but it could be periodic. 
    
    Third, consider a finite gap with periodic endpoints attached to $M$. Now apply the construction from Lemma \ref{lem: MAC to SCM} and Theorem \ref{thm: MAC to SCM Lamination} to $M$ to generate an SCM lamination $\mathcal{S}(M)$ with an SCM $d$-gon $P(M)$. This would cause $M$ to now have an SCM $d$-gon on one side and a polygon on the other. All of the points in question are of the same period by assumption and should return by the identity. But, we see $M$ is interior to a larger polygon with at least $d+1$ sides that returns to itself by the identity contradicting Kiwi's theorem \cite{Kiwi:2001}. 
    
    We are left with the fact that the other side of $M$ contains no gap which means there are only leaves limiting to $M$ on that side. Hence $M$ is a one-sided limit leaf in its MAC lamination. 
    
    Since we have the symmetric sibling portrait, the siblings of $M$ are also one-sided limit leaves by a similar argument. 
\end{proof}

\begin{cor}\label{Cor : Limit Leaves}
Let $\mathcal{L}$ be a MAC lamination with major, $M$, which is identity return. The following hold:
\begin{enumerate}
    \item The leaf, $M$, and its siblings are single sided limit leaves in $\mathcal{L}$. 
    \item There are no isolated leaves in $\mathcal{L}$.
\end{enumerate}
\end{cor}

\begin{proof}
By Section \ref{sec: pullback}, every leaf in $\mathcal{L}$ is either a pullback of the forward orbit of $M$ or a limit of pullbacks of the forward orbit of $M$. By Theorem \ref{thm: MAC Gaps} $M$ is a single sided limit leaf, and pulling $M$ back we have its siblings and their pre-images are also single sided limit leaves giving us (1). Then, we take the closure of the pre-lamination and thus every leaf added is itself a limit leaf giving us (2).
\end{proof}

\section{SCM to Unicritical Laminations}
\subsection{Structure of Canonical MAC Laminations}

Recall in Definition \ref{def: SCM lam} we state an SCM polygon, $P$, is in a maximal critical sector surrounded by $d-1$ critical chords. This will be our starting place to expand our understanding of these polygons in order to find a one-to-one correspondence to MAC laminations. We already have one direction of the correspondence from MAC to SCM. Now, we want to go in the other direction from a canonical SCM lamination to a MAC lamination.

\begin{thm}[Properties of Canonical SCM Laminations] \label{thm: canonical SCM props}
Let $\mathcal{S}$ be a canonical SCM lamination with either a SCM $d$-gon or $k(d-1)$-gon, $P$. Then $\mathcal{S}$ and $P$ have the following properties.
\begin{enumerate}
   
    \item $P$ must have exactly $d-1$ sides of length between $\frac{1}{d}$ and $\frac{1}{d-1}$ and at least one side of length less than $\frac{1}{d}$. \label{Canonical SCM 1}
    \item For each of the $d-1$ longest sides, $\ell_1, ..., \ell_{d-1}$, there is a degree $2$ Fatou gap bounded by the $\ell_i$ and its sibling $\ell'_i$. \label{Canonical SCM 2}
    
    \item In the identity return case, the maximal critical sector has exactly one arc of the circle has two vertices of $P$ lying on it, namely the endpoints of $M$. \label{Canonical SCM 3}
\end{enumerate}
\end{thm}

\begin{proof}
Definition \ref{def: SCM lam} gives us an SCM polygon $P$ and a collection of $d-1$ guiding critical chords each touching one endpoint of each of the $d-1$ longest sides of $P$. For item (\ref{Canonical SCM 1}), each of the sides of $P$ are at their closest approach to criticality (\ref{def:SCM}), meaning each side of $P$ is within $\frac{1}{d(d+1)}$ of critical length $\frac{1}{d}$ (\ref{max cen strip}). Thus, the longest these sides could be is $\dfrac{1}{d}$ + $\dfrac{1}{d(d+1)}$ = $\dfrac{d+2}{d(d+1)}$. Through a simple algebraic argument, we have that $\dfrac{d+2}{d(d+1)} < \dfrac{1}{d-1}$. Thus, the $d-1$ longest sides of $P$ are less than $\dfrac{1}{d-1}$. All of the majors of $P$ are longer than $\frac{1}{d}$ since each subtends a critical chord. There is at least one other side of $P$ that does not have a critical chord that it subtends and is therefore less than $\dfrac{1}{d}$. The sibling $S'$ emanates from the other endpoint of the critical chord. Together $S$ and $S'$ bound a degree two Fatou gap. Item (\ref{Canonical SCM 2}) follows immediately from Definition \ref{def: SCM lam}. Item (\ref{Canonical SCM 3}) follows from the pigeon hole principle. 
\end{proof}

\begin{thm}[MAC and SCM Laminations Compared] \label{thm: MAC and conc SCM compared}
Let $\mathcal{L}(M)$ be a MAC lamination and $\mathcal{S}(M)$ be the corresponding canonical SCM lamination as in Theorem \ref{thm: MAC to SCM Lamination}. In the identity return case, the only difference between $\mathcal{L}$ and $\mathcal{S}$ is the grand orbits of the sides added to $M$ to make $P(M)$. In the rotational case, the grand orbit of the leaves of $\mathcal{L}(M)$ are internal to polygons of $P(M)$. 
\end{thm}

\begin{proof}

In the identity return case, we first want to verify that $\cl{\mathcal{GO}(\partial G)} = \cl{\mathcal{GO}(M)}$ where $G$ is the central gap of $\mathcal{L}(M)$. The boundary leaves of $G$ are made entirely of pre-images of $M$ by Lemma \ref{Top1}, thus we know that pulling back $\partial G$ is the same as pulling back $M$. Therefore $\cl{\mathcal{GO}(\partial G)} =  \cl{\mathcal{GO}(M)}$, which means the two grand orbits give us the same lamination $\mathcal{L}$. By a similar argument $\mathcal{S} = \cl{\mathcal{GO}(P(M))}$. 

Next, we need to verify that $\mathcal{S}$ and $\mathcal{L}$ have the same limit leaves. We begin by noting that $\mathcal{GO}(M) \subset \mathcal{GO}(P(M))$ which means $\mathcal{L} \subset \mathcal{S}$. Now we must show that any limit leaf in $\mathcal{S}$ is also in $\mathcal{L}$. The SCM $d$-gon $P(M)$ is contained in the central gap $G$ of $\mathcal{L}$. Then, $\mathcal{GO}(P(M)) \subset \mathcal{GO}(G)$ because $P(M)$ subdivides $G$. This means there are no limit leaves of $P(M)$ lying outside the $\mathcal{GO}(G)$ that are not already in $\mathcal{L}$. That only leaves the possibility of new limit leaves appearing interior to $\mathcal{GO}(G)$ in $\mathcal{S}$. Consider a limit leaf $\ell$ interior to $\partial G$. The leaf $\ell$ must lie in a degree two Fatou gap of $\mathcal{S}$, or else it is in $\mathcal{GO}(P(M))$ by Theorem \ref{thm: canonical SCM props}. If $\ell$ connects end caps then $S$ and $S'$ do not bound the degree 2 Fatou gap, contradicting $\mathcal{S}$ being canonical by part (\ref{Canonical SCM 2}) of Theorem \ref{thm: canonical SCM props}. Therefore, $\ell$ must be in one end cap of the Fatou gap limited to by pre-images of $P(M)$. The forward images of $\ell$ are limits of the forward images of these pre-images of $P(M)$ including $P(M)$ itself. This reduces to the case where $\ell$ subdivides a Fatou gap. Hence all of the limit leaves in $\mathcal{S}$ are contained in $\mathcal{L}$.

The rotational case follows from Lemma \ref{lem: MAC to SCM}.
\end{proof}

\begin{thm}[Canonical SCM to MAC Laminations] \label{thm: canonical SCM to MAC lam}
Let $\mathcal{S}$ be a canonical SCM lamination containing an SCM polygon, $P$, that is either identity return or rotational. 

Then, there exists a unique corresponding MAC lamination, $\mathcal{L}$, with a leaf $M$ as its MAC leaf. 

If it is identity return, our SCM polygon has $M$ as its shortest side. 
If it is rotational, $M$ is internal to the SCM polygon.

\end{thm}

The following is an outline of the proof of the canonical SCM lamination to MAC lamination theorem which will be carried out through a sequence of lemmas.
\begin{itemize}
    \item First, in the rotational case, we recover $M$ and add it and all of its pullbacks back into the lamination. In the identity return case, $M$ is already a side of $P$. 
    
    \item We will remove the grand orbit of $P\setminus M$, meaning we will get rid of all of the pullbacks and forward images that are not in the grand orbit of $M$.
    
    \item Next, we need to show that this `reduced' version of our subset of the SCM lamination is still a lamination.
    
    \item Lastly, after ensuring what we have left is still a lamination we want to show it is a MAC lamination where our side $M$ is the major leaf.
\end{itemize}

\begin{lem}[Recovering $M$]\label{lem: recovering M}
    Let $\mathcal{S}(M)$ be a canonical SCM lamination and $P(M)$ the corresponding polygon (identity return or rotational). We can recover our MAC major, $M$, defining a $d$-invariant lamination $\widehat{\mathcal{S}}$. Furthermore, $\widehat{\mathcal{S}} \supset \mathcal{S}$. 
\end{lem}

\begin{proof}
In both cases, our goal it to connect our chain of major leaves to get our desired leaf, $M$. Recall that an SCM polygon has $d-1$ adjacent major leaves (Definition \ref{def: SCM lam}).

In the identity return case, our SCM polygon only has one other leaf on its boundary which already connects our chain of major leaves. This is our leaf, $M$. 

In the rotational case, our SCM polygon will have, in addition to the majors, the full forward orbit of each of those major leaves as its sides. As before, we will connect our chain of major leaves by leaf. In this case, this leaf will be internal to our SCM polygon. We add this leaf and its full forward and backward orbit to our lamination. We will denote this new lamination as $\widehat{\mathcal{S}}$.
\end{proof}

\begin{rem}
Note in the rotational case, our polygon in $\widehat{\mathcal{S}}$ has interior leaves and will be resolved in subsequent lemmas. This does not change equivalence classes from $\mathcal{S}$.

\end{rem}

The goal of the following lemma is to remove the co-root leaves of $P$.

\begin{lem}\label{lem: pre-lam SCM to MAC}

Let $\widehat{\mathcal{S}}$ be as described in Lemma \ref{lem: recovering M}. Then $\widehat{\mathcal{L}} = \overline{\widehat{\mathcal{S}}\setminus(\mathcal{GO}(P\setminus M))}$ has a central gap $G$ bounded by $M$ and its siblings. Moreover, $\widehat{\mathcal{L}}$ is compatible with an all critical $d$-gon.
\end{lem}

\begin{proof}
Start with $\widehat{\mathcal{S}}$ as in Lemma \ref{lem: recovering M}. Since $P$ was in a maximal critical sector, the only compatible sibling portrait is the symmetric one for $\widehat{\mathcal{S}}$. Now, if we remove $\mathcal{GO}(P\setminus M)$, then the central strip of the disk bounded by $M$ and its siblings has no long leaves in it. We claim the central strip contains a Fatou gap, $G$, bounded by $M$ and its siblings. Because $\widehat{\mathcal{S}}(M)$ was canonical, we know that there are no long leaves subdividing the central strip except for the leaves of $P$ that were removed. In the central strip, there is room for an all critical $d$-gon with each vertex attached to an appropriate endpoint of $M$ and its siblings. This follows from the fact that our majors we removed only left less than $\dfrac{1}{d}$ in our circle for $M$, and all forward images of $M$ are underneath it. Thus, $\widehat{\mathcal{L}}$ has a central gap $G$ bounded by $M$ and its siblings, and $G$ is compatible with an all critical $d$-gon.


\end{proof}

\begin{lem}\label{lem: pre-lam to lam SCM to MAC}
Let $\widehat{\mathcal{L}}$ be obtained from $\widehat{\mathcal{S}}$ by Lemma \ref{lem: pre-lam SCM to MAC}. Then, $\widehat{\mathcal{L}}$ is a $d$-invariant lamination.
\end{lem}

\begin{proof}
To show that $\widehat{\mathcal{L}}$ from Lemma \ref{lem: pre-lam SCM to MAC} is a $d$-invariant lamination we need to show the following properties:
\begin{enumerate}
    \item $\widehat{\mathcal{L}}$ is a lamination (it does not contain crossing leaves).
    \item $\widehat{\mathcal{L}}$ has only finite equivalence classes.
    \item $\widehat{\mathcal{L}}$ is closed.
    \item $\widehat{\mathcal{L}}$ is forward, backward, and sibling invariant.
\end{enumerate}
Our definition of $\widehat{\mathcal{L}}$ in Lemma \ref{lem: pre-lam SCM to MAC} removed leaves from the $d$-invariant lamination $\widehat{\mathcal{S}}$. Since removing leaves does not cause any intersection of leaves we have (1). Similarly, for (2) since $\widehat{\mathcal{S}}$ was $d$-invariant with finite equivalence classes and removing leaves from a lamination can only decrease the size of equivalence classes, $\widehat{\mathcal{L}}$ only has finite equivalence classes. (3) follows by definition of $\widehat{\mathcal{L}}$. In the construction of $\widehat{\mathcal{L}}$, we removed the entire grand orbit of leaves which leaves $\widehat{\mathcal{L}}$ to be forward, backward, and sibling invariant. 

\end{proof}

\begin{lem}\label{lem: Canonical SCM to MAC}
Let $\widehat{\mathcal{L}}$ be the $d$-invariant lamination obtained from $\widehat{\mathcal{S}}$ by Lemma \ref{lem: pre-lam to lam SCM to MAC}. The leaf, $M$, in $\widehat{\mathcal{L}}$ is a MAC leaf, and $\widehat{\mathcal{L}}$ is the MAC lamination $\mathcal{L}(M)$. 
\end{lem}

\begin{proof}
Let $\widehat{\mathcal{L}}$ be the $d$-invariant lamination obtained from $\widehat{\mathcal{S}}$ by Lemma \ref{lem: pre-lam to lam SCM to MAC}. Let us first consider when $M$ is identity return. We know the leaf, $M$, from the construction of $\widehat{\mathcal{L}}$ is identity return since the $d$-gon, $P$, in $\widehat{\mathcal{S}}$ was identity return and $M$ and its siblings bound a gap $G$ containing an all critical $d$-gon by Lemma \ref{lem: pre-lam SCM to MAC}. Rotating the all critical $d$-gon to the vertices of $M$ and its siblings shows that we have an identity return major leaf that approaches an all critical $d$-gon, and by Definition \ref{def: MAC} $M$ is a MAC leaf.

Now that we have $M$ is a MAC leaf, we want to show $\mathcal{L}(M)$ is equivalent to $\widehat{\mathcal{L}}$ giving us that $\widehat{\mathcal{L}}$ is a MAC lamination. Recall $\widehat{\mathcal{L}} = \cl{\mathcal{GO}(M)}$, and $\mathcal{L}(M)$ is the pullback lamination of $M$. These are two equivalent sets of leaves by Section \ref{sec: pullback}.

The rotational case follows with a similar argument.
\end{proof}

\begin{proof}[Proof of Theorem \ref{thm: canonical SCM to MAC lam}]
Lemmas \ref{lem: pre-lam SCM to MAC} -- \ref{lem: Canonical SCM to MAC} give us a uniquely defined construction starting with an arbitrary canonical SCM lamination $\mathcal{S}$ then leading to a MAC lamination $\widehat{\mathcal{L}}$. Since there were no choices at any step of the construction, it is uniquely determined. Thus, the construction proves the theorem. 
\end{proof}

\section{Future Work: Non-canonical SCM to Unicritical Laminations} 

Now that we have a one to one correspondence between canonical MAC laminations and canonical SCM laminations in the identity return and rotational case, we would like to extend this correspondence to the rotation return case as well. In addition, we want to understand how non-canonical SCM and canonical SCM laminations are related. This is future work on {\em tuning} a lamination by another lamination ``inserted in its Fatou gaps." In our construction from SCM to MAC, we showed that no new limit leaves were added. However, we would like to show that the limit leaves in both laminations are exactly the same.

In addition, we would like to extend our results to ``locally unicritical" laminations. In \cite{Blokh:2021}, the idea of ``flower-like sets" suggests that our correspondence would apply in laminations containing such sets. This idea follows from the fact that in a portion of the circle, we are unicritical and invariant which is what we call ``locally unicritical."


\bibliographystyle{amsplain}

\end{document}